\documentclass{amsart}

\title[A Malliavin-Gamma approach to SGMs for random fields]{A Malliavin-Gamma calculus approach to Score Based Diffusion Generative models for random fields}
\author{Giacomo Greco}
\address{Università degli studi di Roma Tor Vergata}
\curraddr{RoMaDS - Department of Mathematics, 00133RM Rome, Italy.}
\email{greco@mat.uniroma2.it}
\thanks{We warmly thank Giovanni Conforti, Lorenzo Dello Schiavo and Domenico Marinucci for many useful discussions. While this work was written GG was associated to INdAM (Istituto Nazionale di Alta Matematica “Francesco
Severi”) and the group GNAMPA. 
GG is supported by the PRIN project GRAFIA (CUP:
E53D23005530006) and MUR Departement of Excellence Programme 2023-2027 MatMod@Tov (CUP: E83C23000330006).}

\input{X2header}

\begin{document}

\begin{abstract}
     We adopt a Gamma and Malliavin Calculi point of view in order to generalize Score-based diffusion Generative Models (SGMs) to an infinite-dimensional abstract Hilbertian setting. Particularly, we define the forward noising process using Dirichlet forms associated to the Cameron-Martin space of Gaussian measures and Wiener chaoses; whereas by relying on an abstract time-reversal formula, we show that the score function is a Malliavin derivative and it corresponds to a conditional expectation. This allows us to generalize SGMs to the infinite-dimensional setting. Moreover, we extend existing finite-dimensional entropic convergence bounds to this Hilbertian setting by highlighting the role played by the Cameron-Martin norm in the Fisher information of the data distribution. Lastly, we specify our discussion for spherical random fields, considering as source of noise a Whittle-Mat\'ern random spherical field.
\end{abstract}

\maketitle

\section{Introduction}

The aim of generative modelling can be summarized as \textit{creating data from noise} \cite{song2021scorebased}. In practice this means aiming at generating new data from an unknown distribution $\data$, having access only to a bunch of samples. Among all existing methods, the most promising ones (and in practice also the most employed)  are 
Score-based diffusion Generative Models (SGMs) \cite{song2021scorebased,NEURIPS2020_4c5bcfec,pmlr-v37-sohl-dickstein15,dhariwal2021diffusion}.
The core mechanism of SGMs involves two primary processes: a forward diffusion process and a reverse generation process. In the forward process, data is progressively corrupted by adding noise, effectively transforming it into a simple, known distribution (e.g., Gaussian noise). This process can be mathematically described by a stochastic differential equation (SDE) that models the gradual addition of noise over time.
The reverse process aims to recover the original data by learning to denoise the noisy data step by step. This is achieved by estimating the score function, which in the Euclidean setting corresponds to the gradient of the log-density of the forward noising process. By leveraging this score function, the model can iteratively refine noisy samples back into data samples that resemble the original data distribution. This reverse process is also governed by an SDE, which, when solved numerically, facilitates the generation of new data samples from pure noise. To be more precise, in the classical Euclidean setting $\bbRD$ the idea is to consider a noising forward diffusion process started in our data samples (distributed according to the unknown $\data$), that  converges to a known (and easy-to-sample) noising target distribution. A standard choice would be choosing a target Gaussian distribution $\cN(0,\lambda)$, whose density equals $\frm \propto\exp(-\nicefrac{ |x|^2}{2\lambda})\De x$, which means considering as forward noising process the Ornstein–Uhlenbeck process
\be\label{eq:OU:standard}
\De \fwdX_s=-\frac12 \fwdX_s\De s+\sqrt{\lambda}\De B_s\,,\quad\text{ with }\fwdX_0\sim\data\,.
\ee
In order to generate new samples from $\data$ we may then consider the backward process $\bwdX_t$ defined for any $t\in[0,T]$ as $\bwdX_t= \fwdX_{T-t}$.
A classical result in the literature about time-reversal for diffusion processes (see for instance \cite{cattiaux2021time}) states that the law of the time-reversal is a weak solution to the stochastic differential equation
\be\label{eq:bwdX_relative_score}
\bec 
\De \bwdX_t  = - \frac12\bwdX_t \De t + \lambda \nabla \log \frac{d\fwdP_{T-t}}{d\frm}(\bwdX_t)\De t + \sqrt{\lambda}\De B_t\\
\bwdX_T\sim \data\,,
\eec
\ee
where $\fwdP_{s}=\cL(\fwdX_s)$ is the law of the forward process at time $s$ and the vector field 
$\nabla \log\nicefrac{d \fwdP_s}{d\frm}$
is often referred to as the \emph{relative score}, since it corresponds to the gradient of the log-likelihood function. We denote with $\bwdP$ the law of the time-reversal process~\eqref{eq:bwdX_relative_score}, so that $\bwdP_s=\fwdP_{T-s}$. In practice, since $\data$ and its noising evolution $\fwdP_s$ are not known, the score is approximated using our samples and by relying on the relation between the relative score and conditional expectations (cf,~\Cref{thm:gamma-score} below). Similarly, since $\bwdP_0=\fwdP_T$ is not known, the backward dynamics is simulated with $\bwdX_0\sim \frm$, since for $T$ large enough $\bwdP_0=\fwdP_T$ is close to the equilibrium. We postpone a more precise description of the algorithm to \Cref{sec:scheme} (see \Cref{alg:sampling}).

\bigskip

The aim of the present article is to generalize this algorithm to an infinite-dimensional setting, by replacing the Euclidean space $\bbRD$ with a (possibly infinite-dimensional) Hilbert space $\rmH$ (e.g.~the space of spherical random fields $\rmH=\LLS$). Indeed in many applications, the data distribution one would like to consider is supported on an infinite-dimensional space, e.g. functions defined over a sphere $\sS$ like for global weather forecast \cite{Price2025Probabilistic, andrae2025continuous}, or for the analysis of Cosmic Microwave Background (CMB) lensing reconstruction  \cite{floss2024Delensing}. More generally, in recent years diffusion models have found applications in many different cosmological problems and we refer the reader to this (far from being exhaustive) list of latest references \cite{adam2022posterior, Legin2023Posterior, Ono2024Debiasing, Schanz2024Stochastic, heurtel2023removing}. Therefore the possibility of applying SGMs beyond imaging generations calls for the introduction and implementation of diffusion models with possibly infinite-dimensional state spaces. A common way of implementing SGMs in infinite-dimensional settings would be considering a discretization of the state space (\ie, an orthogonal projection onto a finite-dimensional subspace of dimension $d$) and then performing the diffusion model in the standard Euclidean space $\bbRD$. However, in doing so, one might suffer from the dimensionality dependence (in the $d\uparrow\infty$ limit there might not be a limiting noising forward process, or the convergence bounds might diverge). For this reason, defining a forward noising process directly on the infinite-dimensional state space would help solving this issues and as we shall see later provide dimension independent convergence bounds. 

\medskip

The core idea behind this paper is using a more abstract point of view, based on  $\Gamma$-calculus and Dirichlet forms, in order to define an appropriate forward noising dynamics. In doing so, we unveil the role played by Malliavin calculus, Malliavin derivatives, and Wiener chaoses, and we combine it with abstract time-reversal formulas for identifying the backward denoising process. Despite we restrict ourselves here to a Gaussian setting (\ie, the noising dynamics converges to a Gaussian measure), we believe that what we present below might be used as a guideline for engineering infinite-dimensional SGMs with different target noises, which one would like to consider taking into account the properties of $\data$ (e.g., sparsity). Then, we employ this general approach in order to get entropic convergence bounds in $\scrH$ divergence.

Lastly, let us simply mention here that the main difficulty for going beyond the pure Gaussian setting in the infinite-dimensional setting might be the lack of explicit known identities for the corresponding noising stochastic process associated to the abstract Dirichlet form. This would prevent any practical application. Moreover, in order to compute the score functional we rely on some explicit properties of the noising process. Therefore, any non-Gaussian generalization that builds upon the present work would still require the explicit knowledge of the noising process. We will address more in depth this important generalization in a follow-up work.

\subsection{The point of view of $\Gamma$-calculus and Dirichlet forms}\label{sub:abstract:finite} Firstly, let us recall some basics about $\Gamma$-calculus and Dirichlet forms in the finite-dimensional setting and try to formalise the previous ideas in a more abstract framework. We will use the ideas presented here as a guideline in our generalization attempt. Let us start by recalling that once the forward dynamics is given via the SDE~\eqref{eq:OU:standard}, one can always consider the corresponding semigroup $P_tf(x)\coloneqq \bbE[f(X_t)|X_0=x]$, its generator $Lf=\frac\lambda2\Delta f-\frac12\langle x,\nabla f\rangle$, the carr\'e du champ operator $\Gamma(f,g)\coloneqq \langle\nabla f,\nabla g\rangle$ and the associated Dirichlet form 
\be\label{eq:def:dir:OU}
\cE(f,g)\coloneqq\frac12\int\Gamma(f,g)\De \frm=\int f(-L)g\De \frm=\frac12\int \langle\nabla f,\nabla g\rangle \De \frm\,.
\ee
However, in order to define a Markov evolution one could also consider the reverse point of view, by firstly considering a Markov triple structure $(E,\frm,\Gamma)$, that is a measure space $(E,\frm)$ and a carr\'e du champ operator $\Gamma$, defined on some appropriate set of functions. For instance for the  Ornstein–Uhlenbeck process we have $E=\bbRD$ and the operator $\Gamma(f,g)=\langle\nabla f,\nabla g\rangle$ acting on the algebra of smooth compactly supported functions $\cA_0=\cC^\infty_c(\bbRD)$. From this one can then define the generator $L$ (acting on $\cA_0$) by imposing the validity of the integration by parts
\bes
\int f(-L)g\De \frm=\frac12\int\Gamma(f,g)\De\frm\,,
\ees
which in our OU setting yields to $2Lf=\lambda \Delta f-\langle x,\nabla f\rangle$ for any $f\in\cA_0$. Moreover,  we can define the Dirichlet form $\cE$ on $\cA_0\times\cA_0$ as in~\eqref{eq:def:dir:OU} and then extend its definition to the Dirichlet domain $\mathcal{D}(\cE)$, \ie, the completion of $\cA_0$ in $\rmL^2(\frm)$ w.r.t.~the norm $\|f\|_\cE\coloneqq \sqrt{\|f\|_2^2+\cE(f)^2}$ (in this setting $(\mathcal{D}(\cE),\norm{\cdot}_\cE)$ corresponds to the Sobolev space $\rmH^1_0(\frm)$). From this we can then construct the domain of the generator $\cD(L)$ as the subset of functions $f\in\cD(\cE)$ for which $\cE(f,g)\leq C(f)\|g\|_2$ for any $g\in\cD(\cE)$ and extend the definition of $L$ on its domain via the integration by parts formula. Once given the generator $L$ and its domain $\cD(L)$, the OU semigroup is formally defined as $P_t=e^{tL}$. For a more comprehensive description of this process we refer the reader to \cite[Chapter 3 and Appendix A]{bakry2013analysis}.

\medskip

\subsection{Main contributions} 
In order to define the noising dynamics, our strategy is to mimic the description given in \Cref{sub:abstract:finite}. We accomplish this in \Cref{sec:forward} where, after fixing a Gaussian target measure $\frm$ with covariance operator $C$, we consider its associated Cameron-Martin space $\cH_\frm$ and introduce the the carr\'e du champ operator $\Gamma$ and the Dirichlet forms as 
\bes
\Gamma(u,v)\coloneqq \langle \bar\nabla u,\bar\nabla v\rangle_{\cH_\frm} \,,\quad\text{ and }\quad\cE(u,v)=\frac12\int \Gamma(u,v)\,\De\frm\,,
\ees
where $\bar \nabla$ is the Malliavin gradient (associated to $\frm$). By imposing the validity of the integration by parts we then compute the generator $\fwdL$ as the operator that on cylindrical functions $u=F(I(f_1),\dots,I(f_k))$ returns
\bes
\fwdL u=-\frac12\sum_{i=1}^k(\partial_iF)(I(\underbar{f}))\,I(f_i)+\frac12\sum_{i,j=1}^k(\partial^2_{ij}F)(I(\underbar{f}))\,\langle f_i,\,f_j\rangle_{\cH_\frm}\,,
\ees
where $I\colon \cH_{\frm }\to\mathcal{R}_{\frm }$ is the canonical isomorphism between the Cameron-Martin space and the corresponding reproducing kernel Hilbert space $\mathcal{R}_{\frm }$  (\ie, the closure of $\rmH^\star$ in $\rmL^2(\rmH,\frm )$) and for any $h\in\cH_\frm$ the element $I(h)\in\rmL^2(\rmH,\frm)$ is also known as Paley–Wiener integral (w.r.t.~$h\in\cH_\frm$). We postpone to~\Cref{sec:forward} a more precise statement. Finally, there we show  that this construction yields to a well defined Markov process with state space $\rmH$, which corresponds to the Ornstein-Uhlenbeck (hereafter OU) process. We show that this uniquely characterizes our forward noising dynamics~$\fwdP$. Moreover, as soon as $\data\ll\frm$ we further deduce that $\fwdP_s\ll\frm$.

In \Cref{sec:back} we borrow from the seminal work of Cattiaux, Conforti, Gentil and Leonard \cite{cattiaux2021time} an abstract time-reversal formula 
and, we compute the generator associated to the backward dynamics $\bwdP$ as
\bes
\bwdL_t u\coloneqq \fwdL u+\frac{\Gamma(\rho_{T-t},\,u)}{\rho_{T-t}}\,,\quad\text{ where }\rho_s=\nicefrac{\De\fwdP_s}{\De \frm}\,.
\ees
Our main theoretical contribution is then showing in \Cref{thm:gamma-score} that the last term appearing in the backward generator, which hereafter we refer to as $\Gamma$-score, corresponds to a conditional expectation
\bes
\frac{\Gamma(\rho_s,\,u)}{\rho_s}(\Phi)=\frac{1}{e^{\nicefrac{s}{2}}-e^{-\nicefrac{s}{2}}}\,\bbE_{\fwdP}[I(\bar\nabla u(X_s))(X_0-e^{-\nicefrac{t}{2}}X_s)|X_s=\Phi]\,.
\ees
Up to our knowledge, this link between score and conditional expectation was known solely in the finite-dimensional setting. Moreover, our proof highlights the importance and sufficiency of having a commutation identity between Malliavin derivative and OU semigroup (cf.~\Cref{lemma:comm:nabla:Pt} below), which provides a novel useful insight also for the standard Euclidean case. The standing assumption of \Cref{sec:back} is a finite entropy (a.k.a. Kullback-Leibler divergence) condition, that is we assume $\scrH(\data|\frm)$  to be finite, where
\bes
\scrH(\rmP|\rmQ)\coloneqq \begin{cases}
    \bbE_\rmP\biggl[\log\frac{\De\rmP}{\De\rmQ}\biggr]\quad &\text{if }\rmP\ll\rmQ\,,\\
    +\infty \quad&\text{otherwise.}
\end{cases}
\ees
This condition is necessary in order to rely on \cite{cattiaux2021time}, when identifying the backward generator $\bwdL$. Moreover, it is necessary when considering convergence bounds in $\scrH$-divergence.

The link with conditional expectations allows us to mimic the training of finite-dimensional SGMs and properly define the generative model scheme. We explain this in~\Cref{sec:scheme} where we further extend existing $\scrH$-convergence results to our Hilbertian setting. In particular in \Cref{thm:KL:convergence:bound} we prove that, for any finite-dimensional marginalization $J=\{\ell_1,\dots,\ell_N\}$ of our random field ($\Phi\mapsto\Phi^J\in\rset^N$), whenever we are able to learn the score up to an error $\varepsilon$ (cf.~\Cref{ass:score:goodness} below),
the law of the generated sample $Y_{\mathrm{gen}}$ satisfies
\bes
\scrH(\data^J|\cL(Y^J_{\mathrm{gen}}))
    \leq e^{-\nicefrac{T}{2}}\,\scrH(\data|\frm)+T\,\varepsilon^2+2\,h\,\max\{4,h\} \,I(\data|\frm)\,,
\ees
where $h$, denotes the time discretization step of our sampling scheme (same as in Euler-Maruyama schemes for SDE sampling), whereas $I(\data|\frm)$ denotes the Fisher information of $\data$ w.r.t.~$\frm$ which is defined as $I(\data|\frm)\coloneqq \bbE_{\frm}\norm{\bar\nabla\sqrt{\nicefrac{\De\data}{\De\frm}}}^2_{\cH_\frm}$, with $\bar\nabla$ and $\cH_\frm$ being the Malliavin gradient and the Cameron-Martin space associated with the Gaussian measure $\frm$.
Lastly, notice that the bound we prove in \Cref{thm:KL:convergence:bound} is uniform in the finite-dimensional marginalization choice $J$, and therefore might be used in order to prove a $\scrH$-convergence bound at the level of random fields. We leave this question for future works.

Lastly, we conclude the paper with \Cref{sec:example} where we consider as state space the space of spherical random fields and where we consider as noisy Gaussian measure the law of a Whittle-Mat\'ern random field.

\bigskip

\subsection{Comparison with existing literature.}
A vast body of literature in recent years has dealt with the engineering, implementation and convergence of  score-based diffusion generative models. It is not our intention here to fully cover this vast body, but we rather aim at comparing our paper with the latest existing ones, particularly  the ones that address the infinite-dimensional setting, the ones that uses similar techniques or that inspired this work.

As regard the infinite-dimensional setting, up to our knowledge all attempts uses as noising forward dynamics an Ornstein-Uhlenbeck process (as ours) converging to a Gaussian measure with trace class covariance operator. This is the case for instance of \cite{hagemann2023multilevel, lim2023score, Pidstrigach2024Infinite}. However, in \cite{hagemann2023multilevel} the authors treat the time-reversal only for the finite-dimensional marginalization (\ie, they consider the usual Euclidean time-reversal in $\bbRD$, after projecting the forward dynamics onto a $d$-dimensional subspace), whereas in \cite{lim2023score} the authors do not work directly on the forward-backward SDE, but rather introduce ex ante as score functional the Fr\'echet derivative of the log of $\rho_t$ and sample using, without any explicit link to time-reversal (our work here can be seen as a theoretical guarantee of their choice, since we show that their score corresponds to the $\Gamma$-score field, deduced here via the abstract time-reversal formula). Lastly, \cite{Pidstrigach2024Infinite} works directly within the frame of infinite-dimensional forward-backward SDEs, by considering as forward noising process the OU process and then by showing that the strong solution to the SDE driven by conditional expectations has the same law of the time-reversal of the forward process.
Our work can therefore be thought as bridging between \cite{lim2023score} and \cite{Pidstrigach2024Infinite} since we show via an abstract time-reversal formula that the score field is the same one considered in the former and we further identify it with the conditional expectations, connection that was missing in the latter.
A key feature that distinguishes our work from what was previously done in these contributions is the use of Dirichlet forms, $\Gamma$-calculus and Malliavin calculus. Indeed, this allows us to introduce the noising dynamics in an abstract way and we have also shown, given that, how one can deduce the corresponding denoising dynamics by considering a $\Gamma$-calculus time-reversal formula borrowed from \cite{cattiaux2021time}. This paves the way to the introduction of infinite-dimensional score based diffusion models with noising process converging to a possibly non-Gaussian noise. Indeed, though this is common in finite-dimensional settings, the lack of theoretical results in the infinite-dimensional setting restricted so far their applicability to the Gaussian case. The procedure explained here can in principle be used to introduce non-standard noising procedure as well as to deduce the corresponding denoising process. This extension is  currently under investigation and we will address it in a forthcoming work.

It is worth mentioning that, concurrently on the preparation of this manuscript, in \cite{ren2025unified} the authors also noticed that SGMs and more generally  (finite-dimensional) denoising Markov models can be put within the same $\Gamma$-calculus framework leveraging the results of \cite{cattiaux2021time}. Even though their work focuses on finite-dimensional settings, still they illustrate how this $\Gamma$-calculus framework is versatile and it includes denoising Markov models employing geometric Brownian motion and jump processes
as forward dynamics. As regards the use of Malliavin calculus for SGMs, we should mention the very recent work \cite{mirafzali2025malliavincalculusscorebaseddiffusion} 
that appeared while this paper was in preparation, despite it is limited again to  finite-dimensional settings. Regarding possible applications of Malliavin calculus, it is worth mentioning also the recent work \cite{pidstrigach2025conditioning} that focuses in stochastic optimal control and conditional generative modelling, whose aim can be summarized as
 modifying a reference diffusion process to maximize a given terminal-time reward.

Lastly, the spherical random field example we consider at the end of this paper can be compared with the existing Manifold Diffusion Field generative model considered in \cite{elhag2024manifold}, whose idea is to project the random field over an orthonormal basis of eigenfunctions of the Laplace-Beltrami operator (in our example the spherical harmonics, that are the eigenfunctions of the Laplacian operator on the sphere $\Delta_\sS$) and then perform the usual OU finite-dimensional SGMs on the corresponding (truncated) vector of coordinates (in our example the $a_{\ell,m}$ coefficients).

\medskip

While this paper was under review, two recent contributions \cite{mirafzali2025malliavincalculusapproachscore,mirafzali2025scorebaseddiffusionmodelsinfinite} have appeared that are closely related to ours. In the first, the authors combine Malliavin derivatives and divergences with a new Bismut-type formula in order to deduce exact, closed-form, expression for the score function for a broad
class of non-linear diffusion generative models in finite-dimensional settings. In the latter, the authors manage to extend their work from finite- to infinite-dimensional diffusion models, using techniques for differentiating Hilbert-valued processes together with infinite-dimensional extensions of the
Bismut–Elworthy–Li formula. In particular, they obtain closed-form expressions for the score function as well in broadly general settings.

\medskip

As regards the analyses of the convergence of SGMs, in the contributions mentioned above,  authors generally focus on proving convergence bounds in Wasserstein distances, usually by requiring the Lipschitzianity of the score function or the boundedness for the support of the data distribution $\data$. Here we provide an entropic convergence bound which is a straightforward application of the strategy put forth in \cite{Conforti2025KLConvergence}. The same strategy has already been employed and generalized in  \cite{strasman2025an} for analyzing the convergence of finite-dimensional SGMs with noise schedules. It would be impossible to take into account all the convergence bounds proven in the finite settings, however we would like to conclude mentioning these two recent works \cite{gentiloni2025beyond, bruno2025wasserstein} where the authors manage to prove Wasserstein convergence bounds beyond the usual regularity assumption made on the score and on the log-concavity of $\data$.


\section{Noising dynamics in Hilbertian setting}\label{sec:forward}

The state space of our Markov forward noising diffusion is the Hilbert space $(\rmH,\langle\cdot,\cdot\rangle_\rmH)$ (as running example we will think of $\rmH=\LLS$ equipped with its scalar product, induced by the volume measure $\vol$).
We identify the dual of $\rmH$ (namely, $\rmH^\star$) with the Hilbert space itself via Riesz Theorem and denote with $\varphi^\star\in\rmH^\star$ the bounded linear functional $u\mapsto\langle \varphi,u\rangle_{H}$ on $\rmH$, associated to the element $\varphi\in\rmH$.
Next, we should introduce our noising Markovian evolution. This can be done by firstly considering a Markov triple structure $(E,\frm,\Gamma)$, that is a measure space $(E,\frm)$ and a carr\'e du champ operatpr $\Gamma$, defined on some appropriate set of functions. The measure space considered here is the Hilbert space $\rmH$ equipped with the Gaussian measure $\frm$  with symmetric positive definite covariance operator $C$. We assume $C$ to be of trace class, that is we assume $C$ to have the spectral decomposition $CY_\ell=C_\ell Y_\ell$, with $(Y_\ell)_{\ell\geq 0}$ being an orthonormal base of $\rmH$, with $C_\ell>0$ and $\sum_{\ell}C_\ell<\infty$. As running example in $\rmH=\LLS$, the reader could consider $\frm$ being the law of a Whittle-Mat\'ern spherical random field,  $(Y_\ell)_{\ell\geq 0}$ being the spherical harmonics and $(C_\ell)_{\ell\geq 0}$ being the angular power spectrum of the Whittle-Mat\'ern field, cf.~\Cref{sec:example} below.

As concerns the $\Gamma$-operator we should fix an algebra of smooth functions defined on $\rmH$, in analogy to the Euclidean setting where we have considered the algebra of smooth compactly supported functions. In view of that, let us consider the set of cylindrical  functions  
\bes
\PC\coloneqq\biggl\{p=F(I(f_1),\,\dots,\,I(f_k))\quad\colon \quad \begin{aligned}
    \,&\text{for some }k\geq 1,\,F\in\cC^\infty_\mathrm{Pol}(\rset^k)\\
    \,&\text{and }f_i\in\cH_{\frm }\quad\forall\,i=1,\dots,k
\end{aligned}\biggr\}\,,
\ees
where $\cC^\infty_\mathrm{Pol}(\rset^k)$ is the set of $k$-multivariate smooth functions with polynomial growth and with partial derivatives with polynomial growth. The set $\cH_{\frm }$ is the Cameron-Martin space associated to $(\rmH,\frm )$, that is the Hilbert space 
\bes
\cH_{\frm }\coloneqq\biggl\{h\in\rmH\quad\text{s.t. } \sum_{\ell\geq 0} C_\ell^{-1}\,|\langle h,Y_\ell\rangle_H|^2<+\infty\biggr\}\,,
\ees
with scalar product 
\bes
\langle h,k\rangle_{\cH_{\frm }}\coloneqq\langle C^{-\nicefrac12}h, C^{-\nicefrac12}k\rangle_\rmH=\sum_{\ell\geq 0}\frac{\langle h,Y_\ell\rangle_H\langle k,Y_\ell\rangle_H}{C_\ell}\,.
\ees
The functional $I\colon \cH_{\frm }\to\mathcal{R}_{\frm }$, appearing inside the cylindrical function, is the canonical isomorphism between the Cameron-Martin space and the corresponding reproducing kernel Hilbert space $\mathcal{R}_{\frm }$  (\ie, the closure of $\rmH^\star$ in $\rmL^2(\rmH,\frm )$).\footnote{For any $h\in\cH_\frm$ the element $I(h)\in\rmL^2(\rmH,\frm)$ is also known as Paley–Wiener integral (w.r.t. $h\in\cH_\frm$). On the elements in the range of the covariance operator $C$, the integral $I(h)$ is defined via $I(h)(\cdot)=(C^{-1}h)^\star(\cdot)=\langle C^{-1} h,\,\cdot\rangle_{H}$. }  The collection $(I(h))_{\{h\in\cH_{\frm }\}}$ can be thought as an isonormal Gaussian process over $\cH_{\frm}$ defined on the probability space $(\rmH,\sigma(I),\frm )$ with covariance structure $\bbE[I(h)I(k)]=\langle h,k\rangle_{\cH_{\frm }}$ (cf. Proposition 4.40 in~\cite{hairer2023spdesNOTES}). 
    For exposition's clarity, a cylindrical function $p\in\PC$ applied to any element $\varphi\in\rmH$  returns
\bes
p(\varphi)=F(I(f_1)(\varphi),\dots,I(f_k)(\varphi))\,,
\ees
and for notations' sake we will also summarize the above display by simply writing $p=F(I(\underbar{f}))=F\circ I(\underbar{f})$.  
Clearly  we have $\mathcal{R}_{\frm }\subset\PC\subset \rmL^2(\rmH,\frm )$.  
 Moreover, $\PC$ is an algebra of nowhere-vanishing functions that separates points which is dense in $\rmL^2(\rmH,\frm )$, the latter property being a straightforward consequence of the Wiener chaos decomposition \cite[Theorem 1.1.1]{Nualart2006}.

For any cylindrical function $p\in\PC$ as above, we introduce its Malliavin derivative (w.r.t. to $\cH_\frm$) as  
\be\label{eq:gradient:mall:PC}
\bar\nabla p=\sum_{i=1}^k(\partial_i F)(I(f_1),\dots,I(f_k))\,f_i\,\in\cH_\frm\,.
\ee
Notice that for any $f\in\cH_{\frm }$ we immediately have $\bar\nabla I(f)=f\in\cH_{\frm }$ and recall that the gradient operator $\bar\nabla$ here defined is closable from $\rmL^p(\rmH,\frm)$ to $\rmL^p(\rmH;\cH_\frm)$ for any $p\geq 1$ \cite[Proposition 2.3.4]{NourdinPeccati2012book}.
Then we may finally define the carr\'e du champ operator $\Gamma$ on $\PC$ as
\be\label{eq:def:Gamma}
\begin{aligned}
\Gamma(F(I(\underbar{f})),G(I&(\underbar{g})))\coloneqq\,\langle\bar\nabla (F\circ I(\underbar{f})),\,\bar\nabla (G\circ I(\underbar{g}))\rangle_{\cH_\frm}\\
=&\,\sum_{i=1}^k\sum_{j=1}^h (\partial_i F)(I(f_1),\dots,I(f_k))\,(\partial_j G)(I(g_1),\dots,I(g_h))\,\langle f_i, \,g_j\rangle_{\cH_\frm}
\end{aligned}
\ee
for any $F\in\cC^\infty_\mathrm{Pol}(\rset^k)$, $G\in\cC^\infty_\mathrm{Pol}(\rset^h)$ and $\underbar{f}=(f_1,\dots,f_k)\in (\cH_{\frm })^k $ and $\underbar{g}=(g_1,\dots,g_h)\in (\cH_{\frm })^h$. Notice that for any $f,\, g\in\cH_\frm$ we  simply have $\Gamma(I(f), I(g))=\langle f,g\rangle_{\cH_\frm}$.

The above choices induce on $\PC\times\PC$ the positive linear form 
\bes\begin{aligned}
\cE(F(I(\underbar{f})),G(I(\underbar{g})))=&\,\frac12\int_{\rmH}\Gamma(F(I(\underbar{f})),G(I(\underbar{g})))\De \frm \\
=&\,\frac12\bbE\biggl[\left\langle\bar\nabla(F\circ I(\underbar{f}))(\Phi),\bar\nabla(G\circ I(\underbar{g}))(\Phi)\,\right\rangle_{\cH_\frm}\biggr]
\end{aligned}
\ees
where in the last display $\Phi$ is a Gaussian random variable on $\rmH$ with law $\frm$. Moreover, this allows us to introduce the operator $L$ on $\PC$, as the symmetric operator 
defined via the integration by parts formula for any $F\circ I(\underbar{f}),\,G\circ I(\underbar{g})\in\PC$ as
\bes
\frac12\int_{\rmH} (G\circ I(\underbar{g}))\,L(F\circ I(\underbar{f}))\,\De \frm  =-\cE(F(I(\underbar{f})),G(I(\underbar{g})))\,.
\ees
If we fix $w\in\cH_{\frm  }$ and take $F\circ I(\underbar{f})=I(w)$, the right hand side of the above display reads as
\be\label{eq:IbP:singleton}
\begin{aligned}
\cE(I(w), G(I(\underbar{g})))=\frac12\bbE\biggl[\left\langle\bar\nabla(G\circ I(\underbar{g}))(\Phi),\,w\,\right\rangle_{\cH_\frm}\biggr]=\frac12\bbE[(G\circ I(\underbar{g}))(\Phi)\,I(w)(\Phi)]\,,
\end{aligned}\ee
where in the last step we have relied on the integration by parts formula for Malliavin calculus \cite[Lemma 1.2.1]{Nualart2006}.

 Therefore, we may define the operator $L$, at least on the reproducing kernel Hilbert space $\mathcal{R}_\frm$, as $L I(w)=-I(w)/2$. 
 In order to extend it to the whole algebra $\PC$ it is then enough imposing the chain rule/diffusion property (inherited from $\Gamma$)
\be\label{eq:L:PC}
\begin{aligned}
L(F\circ 
 I(\underbar{f}))=&\,\sum_{i=1}^k(\partial_iF)(I(\underbar{f}))\,L I(f_i)+\frac12\sum_{i,j=1}^k(\partial^2_{ij}F)(I(\underbar{f}))\,\Gamma(I(f_i),I(f_j))\\
 =&\,-\frac12\sum_{i=1}^k(\partial_iF)(I(\underbar{f}))\,I(f_i)+\frac12\sum_{i,j=1}^k(\partial^2_{ij}F)(I(\underbar{f}))\,\langle f_i,\,f_j\rangle_{\cH_\frm}\\
=&\,-\frac12\,I(\bar\nabla(F\circ I(\underbar{f}))+\frac12\sum_{i,j=1}^k(\partial^2_{ij}F)(I(\underbar{f}))\,\langle f_i,\,f_j\rangle_{\cH_\frm}\,.
\end{aligned}
\ee

\begin{remark}
Let us notice that the integration by parts step applied in~\eqref{eq:IbP:singleton} can be equivalently deduced from the Cameron-Martin Theorem. We sketch here the formal proof since it highlights the importance of the Cameron-Martin space $\cH_\frm$. The main idea is the simple observation (see \cite[Definition 1.2.1]{Nualart2006} and following discussion) that
\bes
\left\langle\bar\nabla(G\circ I(\underbar{g}))(\Phi),\,w\,\right\rangle_{\cH_\frm}=\lim_{t\downarrow 0} \frac1t[G\circ I(\underbar{g})(\Phi+tw)-G\circ I(\underbar{g})(\Phi)]\,.
\ees
Combining this with the change of variable induced  by the shift function $e^{tw}.f(x)\coloneqq f(x+tw)$ and with the Cameron-Martin Theorem
\bes
\bbE[G\circ I(\underbar{g})(\Phi+tw)]=\int (G\circ I(\underbar{g}))\,\biggl(\frac{\De (e^{tw}.)_{\#}\frm  }{\De\frm  }\biggr)\De\frm =\int (G\circ I(\underbar{g}))\,e^{I(t\,w)-\frac{\norm{t\,w}^2_{\cH_{\frm  }}}{2}}\,\De \frm\,,
\ees
 at least formally, yields to
\bes
\begin{aligned}
\cE(I(w), G(I(\underbar{g})))=&\,\frac12\bbE\biggl[\left\langle\bar\nabla(G\circ I(\underbar{g}))(\Phi),\,w\,\right\rangle_{\cH_\frm}\biggr]=\frac12\bbE[\De_t|_{t=0}\, (e^{tw}.\,G\circ I(\underbar{g}))]\\
=&\,\frac12\int (G\circ I(\underbar{g}))\,\biggl(\De_t|_{t=0} \frac{\De (e^{tw}.)_{\#}\frm  }{\De\frm  }\biggr)\De\frm\\
=&\,\frac12\int (G\circ I(\underbar{g}))\,I(w)\,\De \frm
=\frac12\bbE[(G\circ I(\underbar{g}))\,I(w)] \,.
\end{aligned}\ees
This approach also explains why we are considering the rigging of $\rmH$ with the Cameron-Martin space $\cH_\frm$. Indeed, according to Cameron-Martin Theorem, the shifted measure $(e^{h}.)_{\#}\frm $ is absolutely continuos w.r.t. $\frm$ if and only if $h\in\cH_\frm$, for which we then have
\bes
\frac{\De(e^{h}.)_{\#}\frm  }{\De \frm  }(\varphi)=\exp\biggl(I(h)(\varphi)-\nicefrac{\norm{h}^2_{\cH_{\frm  }}}{2}\biggr)\,.
\ees
\end{remark}

\medskip

In conclusion, by imposing the validity of the integration by parts formula we have defined $L$ on the cylindrical test functions via \eqref{eq:L:PC}. This is enough since $L(\PC)\subset \PC$, which allows also to prove that the corresponding Dirichlet form $\cE$ is closable 
(cf. \cite[Section 3.1.4]{bakry2013analysis}),
and can be extended on its domain $\cD(\cE)$ which is the completion of $\PC$ w.r.t. the norm $\|p\|_\cE\coloneqq[\|p\|^2_{\rmL^2(\rmH,\frm)}+\cE(p)]^{\nicefrac12}$. Then, $L$ can be extended, via the integration by parts formula, from $\PC$ to its domain $\cD(L)$ which is defined as the set
\bes
\cD(L)\coloneqq \{p\in\cD(\cE)\eqsp\colon\eqsp\exists\, C(p)>0\text{ s.t. }\cE(p,q)\leq C(p)\|q\|_{\rmL^2(\rmH,\frm)}\eqsp\forall\,q\in\cD(\cE)\}\,.
\ees
This construction is often referred to as the Friedrichs self-adjoint extension of the operator $L$. Indeed, the operator $L$ designed here is always self-adjoint (on its domain) and is the infinitesimal generator of a symmetric semigroup $(P_t)_{t\geq0}$ in $\rmL^2(\rmH,\frm)$ with invariant measure $\frm$ (cf.~\cite[Section 3.1.4]{bakry2013analysis}). On $\cD(L)$ we have
\bes
\partial_t P_t=L\,P_t=P_t\,L\,.
\ees
Moreover, it is positive preserving and sub-Markov (cf~\cite[Section 1.3.5]{bakry2013analysis}). Actually, this semigroup is well-known and it corresponds to  OU process  on $\rmH$ (see for instance \cite[Section 1.4]{Nualart2006}). To see this, is enough noticing that our generator and the OU generator coincide on our cylindrical algebra $\PC$ \cite[Proposition 1.4.4]{Nualart2006}, that $\cD(L)=\mathbb{D}^{2,2}$ (the completion of $\PC$ under the norm $\|\phi\|_{2,2}=(\|\phi\|^2_{\rmL^2(\rmH,\frm)}+\|\bar\nabla\phi\|^2_{\cH_\frm}+\|\bar\nabla^2\phi\|^2_{\cH_\frm^{\otimes 2}})^{\nicefrac12}$) and that $L$ is essentially self-adjoint, \ie, that $\PC$ is a core for $L$ (meaning that  $\PC$ is dense in $\cD(L)$ in the graph norm $\|\cdot\|_{2,2}$) \cite[Corollary 1.5.1]{Nualart2006}. For the construction of this Dirichlet structure we refer the reader also to \cite[Chapter II Section 3]{MaRockner1992book}. Lastly, owing to \cite[Section 1.4]{Nualart2006} we may express $L$ and the corresponding semigroup $(P_t)_{t\geq 0}$ in terms of Wiener chaoses. In view of that, let us consider the Wiener chaos decomposition 
\bes
\rmL^2(\rmH,\frm)=\bigoplus_{n\geq 0}\cH_n\,,
\ees
where the Wiener chaos of order n $\cH_n$ is defined as the closed linear subspace of $\rmL^2(\rmH,\frm)$ generated by the set $\{H_n\circ I(h)\,\colon\,h\in\cH_\frm\text{ s.t. }\norm{h}_{\cH_\frm}=1\}$ with $H_n$ being the nth Hermite polynomial. If $\Pi_n\colon\rmL^2(\rmH,\frm)\to\cH_n$ denotes the orthogonal projection onto the $n^{th}$ chaos, then for any $u\in\rmL^2(\rmH,\frm)$ we have
\be\label{eq:with:chaos}
P_t u=\sum_{n\geq 0}e^{-\nicefrac{nt}{2}}\Pi_n u\,,\quad\text{ and }\quad L u=\sum_{n\geq 0}-\frac{n}{2} \,\Pi_nu
\ee
and $\cD(L)=\{u\in\rmL^2(\rmH,\frm)\,\colon\,\sum_{n\geq 0}n^2\,\norm{\Pi_n u}_{\rmL^2(\rmH,\frm)}^2<+\infty\}$. 
Finally, owing to Mehler's formula \cite[Equation 1.67]{Nualart2006} the semigroup $P_t$ may be extended to $\rmL^1(\rmH,\frm)$.

\medskip

To summarise, we have introduced a Dirichlet structure on $\rmH$ associated to our Gaussian target measure $\frm\in\cP(\rmH)$ which induces a diffusion process $(X_t,\cF_t)_t$ on $\rmH$ (see \cite[Theorem 3.6]{Albeverio1989Classical} or \cite{Schmuland1990}, \cite[Theorem 3.6]{Albeverio1991SDEviaDirichletForms} and \cite{MaRockner1992book}[Chapter IV]). Moreover, if $\rmP$ denotes the law of this process, then $\rmP$-a.e. we have
\be\label{eq:OU:integrata}
 X_t=X_0-\frac12\int_0^tX_s\De s+W^\frm_t\,,
\ee
where $(W^\frm_t)_{t\geq 0}$ is an $(\cF_t)_t$-Brownian motion on $\rmH$ starting at zero with covariance $\langle\cdot,\cdot\rangle_{\cH_\frm}$ (cf. \cite[Theorem 6.10 and Remark 6.8-(ii)]{Albeverio1991SDEviaDirichletForms} or \cite[Theorem II.3.11]{MaRockner1992book}) and for any $t\geq 0$, for any $u\in\rmL^2(\rmH,\frm)$ it holds $P_t u(x)=\bbE[u(X_t)|X_0=x]$ for $\frm$-a.e. $x\in\rmH$. 
Moreover, from Fukushima decomposition \cite[Theorem 4.3 and Remark 4.4-(ii)]{Albeverio1991SDEviaDirichletForms} it follows that $\rmP$ solves the martingale problem $MP(L,\cD(L))$, \ie, that for any $u\in\cD(L)$ the process 
\be\label{eq:MP:fwd}
M^u_t\coloneqq u(X_t)-u(X_0)-\int_0^t Lu(X_s)\De s\quad\forall\,t\geq 0
\ee
is a $\rmP$-martingale.

\medskip

\medskip

Therefore we can build our forward noising dynamics $(\fwdX_t)_{t\geq 0}$ in $\rmH$ that at time $t=0$ starts distributed according to $\data$, which we assume to be absolutely continuous with respect to the Gaussian measure $\frm$, and that evolves according to the OU dynamics~\eqref{eq:OU:integrata}. Now, let us fix a timewindow $T\gg 1$ and let us consider the path space $\Omega=\cC([0,T],\rmH)$. We will denote the law of the forward process with $\fwdP=\cL(\fwdX_{|_{[0,T]}})\in\cP(\Omega)$. For notation consistency, let us denote the forward generator built in this section with $\fwdL$.
Similarly, we can consider as reference measure the law of a diffusion associated to our Dirichlet structure and started at time $t=0$ already distributed according to the Gaussian measure $\frm$. This induces the stationary Markov measure $\rmR\in\cP(\Omega)$, that at any time $t\geq 0$ satisfies $(X_t)_{\#}\rmR=\frm$. Clearly, $\data\ll\frm$ implies $\fwdP\ll \rmR$. 

\begin{remark}\label{remark:smoothing:semigroup}
    In the forthcoming discussion we are going to assume that $\scrH(\data|\frm)$ is finite. This guarantees smoothing effects of the semigroup $(P_t)_{t\geq 0}$. Indeed, if we set $\mu_t=(X_t)_{\#}\fwdP=\data P_t$ and $\rho_t\coloneqq \nicefrac{\De\mu_t}{\De \frm}=P_t\rho_0$, then the former can be seen as the gradient flow on $\cP_2(\rmH)$ of the relative entropy functional \cite[Section 10.4.8]{greenbook} 
    which implies Sobolev regularity for $\log\rho_t$ \cite[Theorem 11.2.12]{greenbook}. More precisely, we may conclude that for a.e. $t>0$ it holds
    $\int\nicefrac{\|\bar\nabla\rho_t\|^2_{\cH_\frm}}{\rho_t}\De\frm<\oo$, \ie, that $\rho_t$ has finite Fisher information (w.r.t.~$\frm$). Notice that this further implies that $\sqrt{\rho_t}\,\in\cD(\cE)$ and $\rho_t\,\in\cD(\cE)$ (via Poincar\'e inequality \cite[Theorem 5.5.1]{Bogachev1998book}).
    Finally, our convergence bounds will further require $\norm{\bar\nabla \sqrt{\rho_0}}_{\cH_\frm}$ to be finite, \ie, that $\sqrt{\rho_0}\in\cD(\cE)$.
\end{remark}


\section{Denoising via abstract time-reversal}\label{sec:back}

In this section we show how the abstract $\Gamma$-calculus framework introduced for the forward dynamics can be used in order to deduce the backward dynamics. More precisely if $\bwdP\coloneqq r^T_{\#}\fwdP\in\cP(\Omega)$ denotes the time-reversed law (under the time reversal $r^T\colon t\mapsto T-t$) then from \cite[Theorem 5.7]{cattiaux2021time} 
we may deduce that $\bwdP$ satisfies (at least on $\PC$) the martingale problem associated to 
\be\label{eq:def:bwdL}
\bwdL_t u\coloneqq \fwdL u+\frac{\Gamma(\rho_{T-t},\,u)}{\rho_{T-t}}\,,
\ee
or equivalently that for any $u\in\PC$ the process
\be\label{eq:MP:back}
\bwdM^u_t\coloneqq u(X_t)-u(X_0)-\int_0^t\bwdL_su(X_s)\De s
\ee
is a local $\bwdP$-martingale. The validity of \cite[Theorem 5.7]{cattiaux2021time} follows from the fact that our noising dynamics is a diffusion (\ie, that we have defined $\Gamma$ in~\eqref{eq:def:Gamma} by imposing the validity of the abstract diffusion property) and that $\PC$ is total in $\rmL^2(\rmH,\frm)$. In the above definition of $\bwdL_t$ the term with the $\Gamma$ operator is defined in a weak sense, \ie, tested against functions $v\in\PC$. Let us further notice here that as soon as $t>0$ the regularising effect of the semigroup yields to $\sqrt{\rho_t}\in \cD(\cE)$ (cf. \Cref{remark:smoothing:semigroup}) and hence the $\Gamma$-score is well defined (this follows from \cite[Lemma 5.3]{cattiaux2021time} combined with the closability of the Malliavin gradient $\bar\nabla$ from $\PC$ to $\rmL^1(\rmH,\frm)$).

\begin{theorem}[$\Gamma$-score identification]\label{thm:gamma-score}
    Assume $\data\ll\frm$. For any $u\in\PC$ and for any $t>0$ we have
    \bes
\frac{\Gamma(\rho_t,\,u)}{\rho_t}(\Phi)=\frac{1}{e^{\nicefrac{t}{2}}-e^{-\nicefrac{t}{2}}}\,\bbE_{\fwdP}[I(\bar\nabla u(X_t))(X_0-e^{-\nicefrac{t}{2}}X_t)|X_t=\Phi]\,.
\ees
\end{theorem}
\begin{proof}
    Let us start by assuming that $u=I(h)$ for some $h\in\cH_\frm$.  Then, by fixing $\varepsilon\in(0,t)$ and writing $\rho_t=P_{t-\varepsilon}\rho_\varepsilon$ and by exploiting the commutation relationship between Malliavin derivative and semigroup (see  \Cref{lemma:comm:nabla:Pt} below), for any $v\in\PC$ 
    we know that
    \bes
    \begin{aligned}
\bbE_\frm\biggl[\frac{\Gamma(\rho_t,\,u)}{\rho_t}(X)\,v(X)\biggr]=&\,\bbE_\frm\biggl[\langle\bar\nabla \rho_t (X),\,h\rangle_{\cH_\frm}\,\frac{v(X)}{\rho_t(X)}\biggr]\\
=&\,e^{-\frac{t+\varepsilon}{2}}\,\bbE_\frm\biggl[\langle P_{t-\varepsilon}\bar\nabla\rho_\varepsilon(X),\,h\rangle_{\cH_\frm}\,\frac{v(X)}{\rho_t(X)}\biggr]\\
=&\,e^{-\frac{t+\varepsilon}{2}}\,\bbE_\frm\biggl[\langle\bar\nabla\rho_\varepsilon(X),\,h\rangle_{\cH_\frm}\, P_{t-\varepsilon}\biggl(\frac{v}{\rho_t}\biggr)(X)\biggr]
\end{aligned}\ees
where the last step follows from the fact that $P_{t-\varepsilon}$ is self-adjoint in $\rmL^2(\frm)$. Next, if $\delta$ denotes the divergence operator associated to $\bar\nabla$ (\ie, the $\rmL^2(\frm)$-adjoint of the operator $\bar\nabla$ \cite[Section 1.3]{Nualart2006}), then from \cite[Lemma 1.3.1 and Proposition 1.3.3]{Nualart2006} we have
 \be\label{eq:gamma:into:sum}
    \begin{aligned}
\bbE_\frm&\biggl[\frac{\Gamma(\rho_t,\,u)}{\rho_t}(X)\,v(X)\biggr]
=e^{-\frac{t+\varepsilon}{2}}\,\bbE_\frm\biggl[\langle\bar\nabla\rho_\varepsilon(X),\,h\rangle_{\cH_\frm}\, P_{t-\varepsilon}\biggl(\frac{v}{\rho_t}\biggr)(X)\biggr]\\
=&\,e^{-\frac{t+\varepsilon}{2}}\,\bbE_\frm\biggl[\rho_\varepsilon(X)\,\delta\biggl(h\, P_{t-\varepsilon}\biggl(\frac{v}{\rho_t}\biggr)\biggr)(X)\biggr]\\
=&\,e^{-\frac{t+\varepsilon}{2}}\,\bbE_\frm\biggl[\rho_\varepsilon(X)\,\biggl(P_{t-\varepsilon}\biggl(\frac{v}{\rho_t}\biggr)(X)\delta(h)(X)-\langle\bar\nabla P_{t-\varepsilon}\biggl(\frac{v}{\rho_t}\biggr)(X),\,h\rangle_{\cH_\frm}\biggr)\biggr]\\
=&\,e^{-\frac{t+\varepsilon}{2}}\,\bbE_{\fwdP_\varepsilon}\biggl[P_{t-\varepsilon}\biggl(\frac{v}{\rho_t}\biggr)(X)\, I(h)(X)\biggr]-e^{-\frac{t+\varepsilon}{2}}\,\bbE_\frm\biggl[\rho_\varepsilon(X)\,\langle\bar\nabla P_{t-\varepsilon}\biggl(\frac{v}{\rho_t}\biggr)(X),\,h\rangle_{\cH_\frm}\biggr],
\end{aligned}\ee
where in the last step we have used the fact that $\delta(h)=I(h)$ on the  Cameron-Martin space and that $\rho_\varepsilon=\nicefrac{\De\fwdP_\varepsilon}{\De\frm}$.

Let us focus on the first term that appears in the last display. From the semigroup property, we have
\bes
\begin{aligned}
\bbE_{\fwdP_\varepsilon}\biggl[P_{t-\varepsilon}\biggl(\frac{v}{\rho_t}\biggr)(X)\, I(h)(X)\biggr]=&\,\bbE_{\fwdP_\varepsilon}\biggl[I(h)(X)\,\bbE_{\rmR}\biggl[\frac{v(X_t)}{\rho_t(X_t)}\,\bigg|X_{\varepsilon}=X\biggr]\biggr]\\
=&\,\bbE_{X\sim\fwdP_\varepsilon}\biggl[\bbE_{\rmR}\biggl[I(h)(X_\varepsilon)\,\frac{v(X_t)}{\rho_t(X_t)}\,\bigg|X_{\varepsilon}=X\biggr]\biggr]\\
=&\,\bbE_{\fwdP}\biggl[\bbE_{\fwdP}\biggl[I(h)(X_\varepsilon)\,\frac{v(X_t)}{\rho_t(X_t)}\,\bigg|X_{\varepsilon}\biggr]\biggr]\\
=&\,\bbE_{\fwdP}\biggl[I(h)(X_\varepsilon)\,\frac{v(X_t)}{\rho_t(X_t)}\biggr]\,,
\end{aligned}\ees
where the last steps follow from the fact that $\rmR$ and $\fwdP$ are defined via the same dynamics, hence they differ only for the initial distribution and particularly both share the same conditional probabilities (from the Markovianity of $P_t$ and the construction of $\fwdP$, the law of $X_t$ conditionally to $X_\varepsilon$ under $\rmR$ is the same as the one under $\fwdP$). 
Moreover, in the
vanishing $\varepsilon$ limit we then have
\be\label{eq:lim:eps:first:term}
\begin{aligned}
\lim_{\varepsilon\downarrow 0}\bbE_{\fwdP}\biggl[I(h)(X_\varepsilon)\,\frac{v(X_t)}{\rho_t(X_t)}\biggr]&\,=\bbE_{\fwdP}\biggl[I(h)(X_0)\,\frac{v(X_t)}{\rho_t(X_t)}\biggr]\\
=&\,\bbE_{X\sim\fwdP_t}\biggl[\frac{v(X)}{\rho_t(X)}\,\bbE_{\fwdP}[I(h)(X_0)|X_t=X]\biggr]\\
=&\,\bbE_\frm[v(X)\,\bbE_{\fwdP}[I(h)(X_0)|X_t=X]]\,.
\end{aligned}
\ee
To be more precise, in order to consider the vanishing $\varepsilon$ limit above, we proceed as follows. Firstly, we rely on the properties of the divergence operator $\delta$ (cf. \cite[Lemma 1.3.1 and Proposition 1.3.3]{Nualart2006}) and write
\bes
\begin{aligned}
\bbE_{\fwdP}\biggl[I(h)(X_\varepsilon)\,\frac{v(X_t)}{\rho_t(X_t)}\biggr]=&\,\bbE_{X\sim\frm}\biggl[\rho_\varepsilon(X)\, I(h)(X)\,\bbE_{\rmR}\biggl[\frac{v(X_t)}{\rho_t(X_t)}\bigg|X_\varepsilon=X\biggr]\biggr]\\
=&\,\bbE_{\frm}\biggl[\langle h,\,\bar\nabla\biggl(\rho_\varepsilon
\,P_{t-\varepsilon}\biggl(\frac{v}{\rho_t}\biggr)\biggr)\rangle_{\cH_\frm}\biggr]\\
=&\,\bbE_{\frm}\biggl[\rho_\varepsilon\langle h,\,\bar\nabla
P_{t-\varepsilon}\biggl(\frac{v}{\rho_t}\biggr)\rangle_{\cH_\frm}\biggr]+\bbE_{\frm}\biggl[\langle h,\,\bar\nabla \rho_\varepsilon
\rangle_{\cH_\frm}\,P_{t-\varepsilon}\biggl(\frac{v}{\rho_t}\biggr)\biggr]\,.
\end{aligned}
\ees
Next, we can apply the Dominated Convergence Theorem  since the commutation estimate \Cref{lemma:comm:nabla:Pt} guarantees that 
\bes
\begin{aligned}
\bbE_{\frm}\biggl[\rho_\varepsilon \abs{\langle h,\,\bar\nabla
P_{t-\varepsilon}\biggl(\frac{v}{\rho_t}\biggr)\rangle_{\cH_\frm}}\biggr]=&\,e^{-\frac{t-\varepsilon}{2}}\bbE_{\frm}\biggl[\rho_t\abs{\langle h,\,\bar\nabla
\biggl(\frac{v}{\rho_t}\biggr)\rangle_{\cH_{\frm}}}\biggr]\\
\leq&\, \bbE_{\fwdP_t}\biggl[\abs{\langle h,\,\bar\nabla v\rangle_{\cH_\frm}}\biggr]+\bbE_{\frm}\biggl[\frac{\abs{\Gamma(\rho_t,\,u}}{\rho_t}\biggr]
\end{aligned}
\ees
 and that
\bes
\begin{aligned}
    \bbE_{\frm}\biggl[\abs{\langle h,\,\bar\nabla \rho_\varepsilon
\rangle_{\cH_\frm}\,P_{t-\varepsilon}\biggl(\frac{v}{\rho_t}\biggr)}\biggr]=e^{\frac{t-\varepsilon}{2}} \bbE_{\frm}\biggl[\abs{\langle h,\,\frac{\bar\nabla \rho_t}{\rho_t}\rangle_{\cH_\frm}\,v}\biggr]\leq e^{\frac{t}{2}}\bbE_{\frm}\biggl[\frac{\abs{\Gamma(\rho_t,\,u)}}{\rho_t}\,|v|\biggr]\,,
\end{aligned}
\ees
 which are both uniformly bounded in $\varepsilon>0$. Therefore, from the Dominated Convergence Theorem and again from the properties of the divergence operator $\delta$ (cf. \cite[Lemma 1.3.1 and Proposition 1.3.3]{Nualart2006}) we may conclude that 
 \bes
\begin{aligned}
\lim_{\varepsilon\downarrow 0}\bbE_{\fwdP}\biggl[I(h)(X_\varepsilon)\,\frac{v(X_t)}{\rho_t(X_t)}\biggr]
=\bbE_{\frm}\biggl[\langle h,\,\bar\nabla\biggl(\rho_0
\,P_{t}\biggl(\frac{v}{\rho_t}\biggr)\biggr)\rangle_{\cH_\frm}\biggr]
=\bbE_{\fwdP}\biggl[I(h)(X_0)\,\frac{v(X_t)}{\rho_t(X_t)}\biggr]\,,
\end{aligned}
\ees
and hence the validity of~\eqref{eq:lim:eps:first:term}.

Let us now focus on the last term appearing in \eqref{eq:gamma:into:sum}. By relying again on the commutation estimate \Cref{lemma:comm:nabla:Pt}, the self-adjointness of the semigroup and on the properties of the divergence operator $\delta$ (cf. \cite[Lemma 1.3.1 and Proposition 1.3.3]{Nualart2006}), as already done before, we have
\bes
\begin{aligned}
\bbE_\frm\biggl[\rho_\varepsilon(X)\,\langle\bar\nabla P_{t-\varepsilon}\biggl(\frac{v}{\rho_t}\biggr)(X),\,h\rangle_{\cH_\frm}\biggr]=e^{-\frac{t+\varepsilon}{2}}\,\bbE_\frm\biggl[P_{t-\varepsilon}\rho_\varepsilon(X)\,\langle\bar\nabla \biggl(\frac{v}{\rho_t}\biggr)(X),\,h\rangle_{\cH_\frm}\biggr]\\
=e^{-\frac{t+\varepsilon}{2}}\,\bbE_\frm\biggl[\rho_t(X)\,\langle\bar\nabla \biggl(\frac{v}{\rho_t}\biggr)(X),\,h\rangle_{\cH_\frm}\biggr]=e^{-\frac{t+\varepsilon}{2}}\,\bbE_\frm\biggl[\frac{v(X)}{\rho_t(X)}\,\delta(h\,\rho_t(X))\biggr]\\
=e^{-\frac{t+\varepsilon}{2}}\,\bbE_\frm[v(X)\,I(h)(X)]-e^{-\frac{t+\varepsilon}{2}}\,\bbE_\frm\biggl[\frac{v(X)}{\rho_t(X)}\,\langle\bar\nabla \rho_t(X),\,h\rangle_{\cH_\frm}\biggr]\\
=e^{-\frac{t+\varepsilon}{2}}\,\bbE_\frm[v(X)\,I(h)(X)]-e^{-\frac{t+\varepsilon}{2}}\,\bbE_\frm\biggl[\frac{\Gamma(\rho_t,\,u)}{\rho_t}(X)\,v(X)\biggr]\,.
\end{aligned}
\ees

Putting together this last display with \eqref{eq:gamma:into:sum}, \eqref{eq:lim:eps:first:term} and letting $\varepsilon\downarrow 0$ we finally deduce that
\bes
\begin{aligned}
\bbE_\frm\biggl[\frac{\Gamma(\rho_t,\,u)}{\rho_t}(X)\,v(X)\biggr]=&\,e^{-\nicefrac{t}{2}}\,\bbE_\frm[v(X)\,\bbE_{\fwdP}[I(h)(X_0)|X_t=X]]\\
&-e^{-t}\bbE_\frm[v(X)\,I(h)(X)]
+e^{-t}\,\bbE_\frm\biggl[\frac{\Gamma(\rho_t,\,u)}{\rho_t}(X)\,v(X)\biggr],
\end{aligned}
\ees
and hence that
\be\label{eq:formula:gamma:monomi:semplici}
\bbE_\frm\biggl[\frac{\Gamma(\rho_t,\,u)}{\rho_t}(X)\,v(X)\biggr]=\frac{1}{e^{\nicefrac{t}{2}}-e^{-\nicefrac{t}{2}}}\,\bbE_\frm\biggl[v(X)\,\bbE_{\fwdP}[I(h)(X_0-e^{-\nicefrac{t}{2}}X_t)|X_t=X]\biggr]\,.
\ee

\medskip

Finally, if we consider a general $u=F\circ I(\underbar{f})\in\PC$, from \eqref{eq:gradient:mall:PC} we may deduce that 
\bes
\begin{aligned}
\bbE_\frm&\biggl[\frac{\Gamma(\rho_t,\,u)}{\rho_t}(X)\,v(X)\biggr]=\bbE_\frm\biggl[\langle\bar\nabla \rho_t (X),\,\bar\nabla u(X)\rangle_{\cH_\frm}\,\frac{v(X)}{\rho_t(X)}\biggr]\\
&=\bbE_\frm\biggl[\langle\bar\nabla \rho_t (X),\,\sum_{i=1}^k (\partial_i F)(I(f_1),\dots,I(f_k))(X)\,f_i\rangle_{\cH_\frm}\,\frac{v(X)}{\rho_t(X)}\biggr]\\
&=\sum_{i=1}^k\,\bbE_\frm\biggl[\langle\bar\nabla \rho_t (X),\,f_i\rangle_{\cH_\frm}\,\frac{v(X)\,(\partial_i F)\circ I(\underbar{f})(X)}{\rho_t(X)}\biggr]\\
&=\frac{1}{e^{\nicefrac{t}{2}}-e^{-\nicefrac{t}{2}}}\,\sum_{i=1}^k\,\bbE_\frm\biggl[v(X)\,\bbE_{\fwdP}[(\partial_i F)\circ I(\underbar{f})(X_t)\,I(f_i)(X_0-e^{-\nicefrac{t}{2}}X_t)|X_t=X]\biggr]\,,
\end{aligned}\ees
where the last step follows from the previous discussion by noticing that $v(X)\,(\partial_i F)\circ I(\underbar{f})(X)\in\PC$ and hence the validity of~\eqref{eq:formula:gamma:monomi:semplici} with this modified test function. Recalling again \eqref{eq:gradient:mall:PC} we finally get
\bes
\bbE_\frm\biggl[\frac{\Gamma(\rho_t,\,u)}{\rho_t}(X)\,v(X)\biggr]=\frac{1}{e^{\nicefrac{t}{2}}-e^{-\nicefrac{t}{2}}}\bbE_\frm\biggl[v(X)\bbE_{\fwdP}[I(\bar\nabla u(X_t))(X_0-e^{-\nicefrac{t}{2}}X_t)|X_t=X]\biggr].
\ees

Since $v\in\PC$ is arbitrary and $\PC$ is dense in $\rmL^2(\rmH,\frm)$ our thesis follows.
\end{proof}

\Cref{thm:gamma-score}  connects the abstract time-reversal formula~\eqref{eq:def:bwdL} and the $\Gamma$-score with the conditional expectation of $X_0$ (conditioned to $\fwdX_t$), which in the finite-dimensional setting is known to correspond to the relative score functional (notice that our result applies also to the finite-dimensional case $\rmH=\bbRD$).  As we will see in the next section, one of the reason why the score function can be efficiently learned is due to this link with conditional expectations (and henceforth with $\rmL^2$-projections).

\begin{remark}
The previous proof highlights the importance of having a commutation relation between the semigroup and the Malliavin gradient. Indeed, this is the only step in the proof where we have explicitly relied on the OU nature of the forward dynamics. This suggests that the approach presented here might be generalized to different noising (diffusion) dynamics, provided a commutation relation of the form $\bar\nabla P_t=f(t)P_t\bar\nabla$ for a general $f(t)\in(0,1)$. Particularly, this further highlights the fact that the score equals a conditional expectation of $X_0$ (conditioned to $\fwdX_t$) provided the commutation between Malliavin gradient and semigroup, and henceforth its validity might be extended beyond the OU dynamics.  
\end{remark}

For sake of completeness, below we state the commutation relation between Malliavin gradient and semigroup employed in the previous proof.

\begin{lemma}[Lemma 1.4.2 in \cite{Nualart2006}]\label{lemma:comm:nabla:Pt}
  For any $u\in\rmL^1(\rmH,\frm)$ and any $t\geq 0$ we have
    \bes
\bar\nabla P_t\,u=e^{-\nicefrac{t}{2}}P_t\bar\nabla u\,.
    \ees
\end{lemma}
\begin{proof}
 If $u\in\PC$ the thesis trivially follows from the explicit expressions~\eqref{eq:with:chaos} combined with \cite[Lemma 1.4.2]{Nualart2006}. The extension to $\rmL^1(\rmH,\mu)$ follows from the density of $\PC\subseteq\rmL^1(\rmH,\frm)$ and the closability of $\bar\nabla$ \cite{NourdinPeccati2012book}, combined with the fact that $(P_t)_{t\geq 0}$ is a linear contraction operator on $\rmL^1(\rmH,\mu)$ \cite[Proposition 2.8.6]{NourdinPeccati2012book}. 
\end{proof}

\medskip

As a corollary of the above identification we may then deduce that the law $\bwdP$ of the time-reversal process weakly satisfies in $[0,T]$ a time-inhomogeneus SDE. We will prove this result under an extra regularity assumption for $\data$, namely a finite entropy condition.

\begin{corollary}\label{cor:back:SDE}
    Assume that $\scrH(\data|\frm)<\oo$. 
    Then the time reversal law $\bwdP$ weakly satisfies on the time interval $[0,T]$ the following SDE
    \be\label{eq:SDE:bwd}
\begin{cases}
    \De \bwdX_s=-\frac12\,\bwdX_s\De s+\sigma_{T-s}^{-1}\,\bbE_{\fwdP}[\fwdX_0-e^{\nicefrac{T-s}{2}}\fwdX_{T-s}|\fwdX_{T-s}=\bwdX_s]\,\De s+\De W^{\cH_\frm}_s\\
    \bwdX_0\sim\fwdP_T\,,
\end{cases}
\ee
where $\sigma_s=e^{\nicefrac{s}{2}}-e^{-\nicefrac{s}{2}}$.
\end{corollary}
\begin{proof}
    Firstly let us notice that the martingale problem \eqref{eq:MP:back} on the test functions $\PC$ with $\bwdL_t$ defined via \eqref{eq:def:bwdL}, with starting marginal $\bwdP_0=\fwdP_T$, admits a unique solution with finite entropy w.r.t.~$\rmR$ (and hence it uniquely characterizes $\bwdP$).
        In view of that, for any fixed $N$ and finite collection $J=\{\ell_1,\dots,\ell_N\}$, let $\mathrm{proj}_J$ denote  the projection on the subspace generated by the eigenfunctions of $C$ with indices in $J$, \ie
    \bes
\mathrm{proj}_J\colon\rmH\to\rset^N\,,\quad \Phi\mapsto\Phi^J\coloneqq (\Phi^{\ell_j}\coloneqq\langle \Phi,\,Y_{\ell_j}\rangle_\rmH)_{j=1,\dots,N}\,,
    \ees
    and for any path measure $\rmQ\in\cP(\Omega)$, with starting marginal $\rmQ_0=\bwdP_0=\fwdP_T$, solving \eqref{eq:MP:back}, consider its pushforward $\rmQ^J\coloneqq(\mathrm{proj}_J)_{\#}\rmQ\in\cP(\cC([0,T],\rset^N))$. 

    \medskip
    
    \noindent \textbf{Claim: $\rmQ^J=\bwdP^J$.} Let us start by showing that $\rmQ^J$ is a solution to the (finite-dimensional) martingale problem associated to the generator $\bwdL_s^{J}$ defined on any test function $u\in\cC^\infty_\mathrm{b}(\rset^N)$ as
\be\label{eq:bMP:finite}
 \bwdL_s^{J} u(x)=\fwdL^{J} u(x)+\sigma_{T-s}^{-1}\sum_{j=1}^N \partial_{\ell_j} u (x)\,\bbE_{\fwdP^{J}}[X_0^{\ell_j}-e^{-\nicefrac{T-s}{2}}X_{T-s}^{\ell_j}|X_{T-s}^{J}=x],
\ee
where $\fwdL^{J}$ is the generator of the $N$-dimensional Ornstein-Uhlenbeck process 
\be\label{eq:OU:marginalization:forward}
\De X_t=-\frac12 X_t\,\De t+ \Sigma^{J}\,\De B_t,\,\quad\text{ with }\Sigma^{J}\coloneqq\mathrm{diag}(C_{\ell_j}^{\nicefrac12})_{j=1,\dots,N}\,.
\ee
whose law $\fwdP^{J}=(\mathrm{proj}_J)_{\#}\fwdP$  coincides with the finite-dimensional marginalisation of $\fwdP$ over the indices in $J$. To see this, introduce the coordinate variables $\phi^{\ell_j}\coloneqq\langle\phi,Y_{\ell_j}\rangle_{\rmH}=I(C_{\ell_j}Y_{\ell_j})(\phi)$ and note that any smooth bounded test function $u\in\cC^\infty_{b}(\rset^N)$  can be written as a cylindrical test function $u^{\mathrm{ext}}(\phi)\coloneqq u(\phi^{\ell_1},\dots , \phi^{\ell_N}\rangle_\rmH)=u\circ I((C_{\ell_j}Y_{\ell_j})_{j=1,\dots,N})(\phi)$ so that from \eqref{eq:def:bwdL}, \eqref{eq:MP:back} and \Cref{thm:gamma-score} it follows
\bes
\begin{aligned}
\bbE_{\rmQ^J}[u(X_t)-u(X_0)]=\bbE_\rmQ[u^{\mathrm{ext}}(X_t)-u^{\mathrm{ext}}(X_0)]\\
\overset{\eqref{eq:MP:back}}{=}
\int_0^t\sigma_{T-s}^{-1}\,\bbE_\rmQ \biggl[\bbE_{\fwdP}[I(\bar\nabla u^{\mathrm{ext}}(\fwdX_s))(\fwdX_0-e^{-\nicefrac{T-s}{2}}\fwdX_{T-s})|\fwdX_{T-s}=X_s]\biggr]\De s\\
+\int_0^t \bbE_{\rmQ}\fwdL u^{\mathrm{ext}}(X_s)\De s\\
\overset{\eqref{eq:L:PC}}{=}\int_0^t\sigma_{T-s}^{-1}
\bbE_\rmQ \biggl[\nabla u(X^J_s)\cdot\bbE_{\fwdP}[\fwdX^J_0-e^{\nicefrac{T-s}{2}}\fwdX_{T-s}^J|\fwdX_{T-s}=X_s]\biggr]\De s\\
+\int_0^t\bbE_\rmQ\biggl[-\frac12\nabla u(X^J_s)\cdot X^J_s+\frac12\sum_{j=1}^NC_{\ell_j}\partial^2_{jj}u(X_s^J)\biggr]\De s\\
=\int_0^t\sigma_{T-s}^{-1}
\bbE_\rmQ \biggl[\nabla u(X^J_s)\cdot\bbE_{\fwdP^J}[\fwdX^J_0-e^{\nicefrac{T-s}{2}}\fwdX_{T-s}^J|\fwdX_{T-s}^J=X_s^J]\biggr]\De s\\
+\int_0^t\bbE_{\rmQ} \fwdL^J u(X^J_s)\De s\,,
\end{aligned}\ees
where the last step follows from the law of total expectations and the tower property since we have the following inclusion of events $\{\fwdX_{T-s}=X_s\}\subseteq\{\fwdX_{T-s}^J=X_s^J\}$. Since under the law $\rmQ$ we have $X^J\sim\rmQ^J$ we may conclude that 
\bes
\bbE_{\rmQ^J}[u(X_t)-u(X_0)]=
\int_0^t \bbE_{\rmQ^J}\bwdL^J u(X_s)\De s\,,
\ees
and that $\rmQ^J$ solves the (finite-dimensional) martingale problem associated to the generator defined in \eqref{eq:bMP:finite}. 
This is enough for proving our claim since 
 since $\scrH(\rmQ^{J}|\rmR^{J})\leq \scrH(\rmQ|\rmR)<\oo$, and this uniquely characterizes $\rmQ^{J}$ as the time reversal of the OU process on $\rset^N$ started in $(\mathrm{proj}_J)_{\#}\data$ \cite[Theorem 5.7]{cattiaux2021time}, and hence that $\rmQ^{J}=\bwdP^{J}$.

 \medskip
 
 \noindent \textbf{Conclusion.} We have shown that $\rmQ^{J}=\bwdP^{J}$, or equivalently that  $(r^T_{\#}\rmQ)^{J}=r^T_{\#}\rmQ^{J}=\fwdP^{J}$.  Since $N\in\N$ and the collection of eigenfunctions $(Y_{\ell_j})_{j=1,\dots,N}$ were arbitrary, by the Kolmogorov Extension Theorem we may conclude that $r^T_{\#}\rmQ=\fwdP$,
and hence that the MP~\eqref{eq:MP:back} uniquely characterizes $\rmQ=\bwdP$. Finally, it is enough noticing that any weak solution of the backward SDE \eqref{eq:SDE:bwd} satisfies the martingale problem \eqref{eq:MP:back} (owing to the explicit expression for the $\Gamma$-score given in \Cref{thm:gamma-score}).

\end{proof}

The previous result is particularly useful in the implementation of the generative model since it gives us a way of simulating $\data$ as the time evolution of the SDE \eqref{eq:SDE:bwd}, provided that we are able to compute (an approximate) version of the conditional expectation appearing in the drift term and provided that we have access to samples from $\fwdP_T$ (or better, from a probability measure close to it, e.g. from the Gaussian $\frm$ which is close enough for $T\gg 1$, see~\eqref{eq:LSI:conv:KL} below). We refer to the next section for a more detailed analysis of the sampling scheme.


\section{$\Gamma$-Score approximation and sampling scheme}\label{sec:scheme}

From the backward SDE~\eqref{eq:SDE:bwd} we see that the $\Gamma$-score functional contribution comes from the drift term 
\bes
\bbE_{\fwdP}[\fwdX_0-e^{\nicefrac{T-s}{2}}\fwdX_{T-s}|\fwdX_{T-s}=\Phi]=\bbE_{\fwdP}[\fwdX_0|\fwdX_{T-s}=\Phi]-e^{\nicefrac{T-s}{2}}\Phi\,.
\ees
Therefore the core of any generative model (based on our OU noising forward dynamics) boils down to learning from noising forward dynamics the conditional expectation function
$\bbE_{\fwdP}[\fwdX_0|\fwdX_{T-s}=\Phi]$.
Owing to the characterization of conditional expectations as $\rmL^2$-projections \cite[Corollary 8.17]{klenke},
we know that for any $\sigma(\fwdX_{T-s})$-measurable random field $Y\in\rmL^2(\fwdP_{T-s})$ 
\bes
\bbE[\|\fwdX_0-Y\|_{\rmH}^2]\leq \bbE[\|\fwdX_0-\bbE[\fwdX_0|\fwdX_{T-s}]\|_{\rmH}^2]\,,
\ees
and equality is attained at $Y=\bbE[\fwdX_0|\fwdX_{T-s}]$. Therefore, given a parametric family of smooth transformations $\{(s,\phi)\mapsto \mathrm{v}_s^\theta(\phi)\}_{\{\theta\in\Theta\}}$ (generally via neural networks) with values in $\rmH$ we may approximate the $\Gamma$-score by minimizing
\bes
\min_{\theta\in\Theta}\bbE [\|\fwdX_0-\mathrm{v}_{T-s}^\theta(\fwdX_{T-s})\|_\rmH^2 ]\,.
\ees
Above we consider a minimization problem over $\Theta$ for each fixed time $s$. In practice we fix  a sequence of times $\{0=t_0<t_1<\dots<t_M=T\}$ (e.g., uniformly sampled over $[0,T]$) and minimize the above expression integrated in time via the discretization and by replacing the expectation with the empirical mean computed on the forward dynamics of our data samples
\be\label{eq:likelihood:scores}
\min_{\theta\in\Theta}\,\frac1N\sum_{i=1}^N\sum_{j=1}^M\frac{(t_j-t_{j{-1}})}{T}\,\|\fwdX^{(i)}_0-\mathrm{v}_{T-t_j}^\theta(\fwdX^{(i)}_{T-t_j})\|_\rmH^2 \,,
\ee
where $(\fwdX_0^{(i)})_{1\leq i\leq N}\overset{\mathrm{iid}}{\sim}\data$ are our original data samples whereas $(\fwdX_s^{(i)})_{1\leq i\leq N}$ correspond to their (independent) time evolutions induced by our Dirichlet structure (\ie, according to \eqref{eq:OU:integrata}).

The training algorithm can be summarized as follows

\medskip

\begin{algorithm}[H]
   \caption{Training}\label{alg:training}
\LinesNumbered
\KwIn{$N$ samples $X^{(i)}\sim\data$ from our data distribution, time horizon $T\gg1$ large enough, partition $0=t_0<t_1<\cdots<t_M=T$ (usually uniformly at random or equidistributed), parametrized vector field family $\{\mathrm{v}^{\theta}:\theta\in\Theta\}$}
\For{$i=1,\dots, N$}{
\tcc{Forward OU evolution  }
Initialize $\fwdX_0^{(i)}=X^{(i)}$\\
Sample the forward OU process $(\fwdX_{t_{j}}^{(i)})_{j=1,\dots,N}$ using \eqref{eq:OU:integrata}
}
Learn $\theta^\star$ optimizer in \eqref{eq:likelihood:scores}\\
\KwOut{$\mathrm{v}^{\theta^\star}$}
\end{algorithm}

\medskip

Once the approximate score $\mathrm{v}^{\theta^\star}$ is learned, on the time-grid $(T-t_j)_{j=0,\dots,M}$, we can (approximately) sample from our data distribution. For simplicity and exposition's sake we consider hereafter $(t_i)_{i=1,\dots, M}$ equidistributed, \ie, with constant stepsize $h=\nicefrac{T}{M}$ (but the generalisation is straightforward). Since $\bwdP_T=\data$ \Cref{cor:back:SDE} guarantees that we can sample from $\data$ by simulating \eqref{eq:SDE:bwd} with Euler-Maruyama method. However, since the $\Gamma$-score is unknown and we do not have access to $\fwdP_T$, we will assume that $\fwdP_T$ is close to the Gaussian reference measure $\frm$ (which eventually holds for $T\gg1$) and we will replace  the $\Gamma$-score with the approximate score $\mathrm{v}^{\theta^\star}$. In conclusion we will simulate the approximate SDE
   \be\label{eq:SDE:approx}
\begin{cases}
    \De Y_s=-\frac12\,Y_s\De s+\sigma_{T-t_j}^{-1}\,(\mathrm{v}^{\theta^\star}_{T-t_j}(Y_{t_j})-e^{\nicefrac{T-t_j}{2}}Y_{t_j})\De s+\De W^{\cH_\frm}_s \quad \forall\,s\in[t_j,t_{j+1}]\\
    Y_0\sim\frm\,,
\end{cases}
\ee

\medskip

\begin{algorithm}[H]
   \caption{Sampling from ($\theta^\star$-approximate) backward evolution}\label{alg:sampling}
\LinesNumbered
\KwIn{$\mathrm{v}^{\theta^\star}$ approximate score, $h$ step size, $T$ time window}
Sample a Gaussian random variable $Y^{\mathrm{EM}}_0\sim\frm=\mathcal{N}(0,C)$\\
\For{$j=1,\dots, M=T/h$}{
\tcc{Euler-Maruyama for approximate backward SDE \eqref{eq:SDE:approx}}
set $\sigma_{j-1}=e^{\frac{T-t_{j-1}}{2}}-e^{-\frac{T-t_{j-1}}{2}}$\\
sample $\xi^j\sim\frm=\mathcal{N}(0,C)$ \\
set $Y^{\mathrm{EM}}_{t_j}\coloneqq Y^{\mathrm{EM}}_{t_{j-1}}-\frac{h}{2}\,Y^{\mathrm{EM}}_{t_{j-1}}+\frac{h}{\sigma_{{j-1}}}\biggl(\mathrm{v}^{\theta^\star}_{T-t_{j-1}}(Y^{\mathrm{EM}}_{t_{j-1}})-e^{\nicefrac{T-t_j}{2}}Y^{\mathrm{EM}}_{t_{j-1}}\biggr)+\sqrt{h}\,\xi^j$
}
\KwOut{$Y^{\mathrm{EM}}_T$}
\end{algorithm}


\subsection{Generative model for spherical random fields}\label{sec:example}
In this section we specify our infinite-dimensional generative model to the spherical random fields case, \ie, when the state space corresponds to the Hilbert space $\rmH=\LLS$. For a comprehensive introduction to spherical random fields we refer the interested reader to the monograph \cite{marinucci2011random}. For sake of completeness we introduce here the objects that will appear below. A spherical random field can be seen as a random element $T\in\LLS$. A nice characteristic is that it can be decomposed in a spectral decomposition (also known as the  Karhunen-Lo\`eve decomposition, which holds almost surely thanks to the deterministic Peter-Weyl Theorem) as
\bes
T(x)=\sum_{\ell=0}\sum_{m=-\ell}^\ell a_{\ell,m} Y_{\ell,m}(x)
\ees
where the $a_{\ell,m}$'s are random coefficients whereas  the deterministic functions $Y_{\ell,m}\in\LLS$ are called spherical harmonics and form an orthonormal basis for $\LLS$ of eigenfunctions of the spherical Laplacian, \ie, that
\bes
\Delta_\sS Y_{\ell,m}=-\ell(\ell+1)Y_{\ell,m}\,. 
\ees
This harmonics are explicitly known and can be expressed in terms of Legendre polynomials. As a consequence, all the randomness of the field is encoded in the sequence
 $(a_{\ell,m})\in\ell^2(\N)$ and we have
\bes
\|T\|_{\rmL^2(\Omega;\LLS)}^2=\bbE\|T\|^2_\LLS=\sum_{\ell=0}^\infty \sum_{m=-\ell}^\ell \bbE |a_{\ell,m}|^2=\|(a_{\ell,m})_{\ell,m}\|_{\ell^2}^2\,.
\ees
Moreover, the coefficients $a_{\ell,m}$ can be computed by inverting the Karhunen-Lo\'eve decomposition as 
\bes
a_{\ell,m}=\int_{\sS}T(x)Y_{\ell,m}(x)\De \vol\,.
\ees

A particular role is played by isotropic random fields  (\ie, the fields whose law is invariant under the action of $SO(3)$, the  group of rotations in $\rset^3$), since under this assumption we further know that
\bes
\bbE[a_{\ell,m}a_{\ell',m'}]=C_\ell\delta^\ell_{\ell'}\delta^m_{m'}
\ees
and the sequence $(C_\ell)_\ell$ is known as angular power spectrum of $T$.
 Finally, if the field is $2$-weakly isotropic, then the spectral decomposition equality holds in the strong $\rmL^2(\Omega;\LLS)$ sense, \ie, that
\bes
\lim_{L\to\infty}\bbE\int_{\sS}\norm{T(x)-\sum_{l=0}^L\sum_{m=-l}^l a_{l,m}Y_{l,m}(x)}^2\De\vol=0\,.
\ees

\bigskip

Therefore our noising dynamics will take place in $\LLS$ and the process that we will consider will be a space-time random field.
As target noise measure we consider the law of a Whittle–Mat\'ern field, which can be formally defined as the solution of the SPDE
\bes
(\kappa^2-\Delta_\sS)^\beta\, u=\cW\,,
\ees
where $\beta,\kappa>0$ are the smoothness and correlation parameters, whereas $\cW$ is a spherical Gaussian white noise, \ie, a generalized  centered Gaussian random field defined on test functions $\phi,\psi\in\rmL^2(\sS)$ as
\bes
\cov{(\cW,\phi)_{\rmL^2(\sS)}}{(\cW,\psi)_{\rmL^2(\sS)}}=(\phi,\psi)_\LLS\,.
\ees
Formally $\cW$ can be seen as the spherical field with angular power spectrum $C^\cW_\ell=1$ for any $\ell\in\N$, and with spectral coefficients $a^\cW_{\ell,m}\sim\cN(0,1)$ (whose associated series does not converge in $\LLS$). From this, the Whittle-Mat\'ern field can be formally computed as follows:
\bes
\begin{aligned}
u=[(\kappa^2-\Delta_{\sS})^\beta]^{-1}\cW=\sum_{\ell=0}^\infty\sum_{m=-\ell}^\ell a^\cW_{\ell,m} [(\kappa^2-\Delta_{\sS})^\beta]^{-1}Y_{\ell,m}\\
=\sum_{\ell=0}^\infty\sum_{m=-\ell}^\ell a^\cW_{\ell,m}(\kappa^2+\ell(\ell+1))^{-\beta} Y_{\ell,m}\,.
\end{aligned}
\ees
Hence the Whittle-Mat\'ern field is  the Gaussian random field associated to the random coefficients $a^{\kappa,\beta}_{\ell,m}\sim\cN(0,(\kappa^2+\ell(\ell+1))^{-2\beta})$ and with angular power spectrum
\be\label{eq:ang:power:spectrum}
C^{\kappa,\beta}_\ell=(\kappa^2+\ell(\ell+1))^{-2\beta}\,.
\ee
For $\beta>\nicefrac12$ the Karhunen-Lo\`eve expansion converges in $\rmL^2(\Omega;\LLS)$ and 
\bes
\|u\|^2_{\rmL^2(\Omega;\LLS)}=\bbE\|u\|^2_\LLS=\sum_{\ell=0}^\infty (2l+1) C_\ell^{\kappa,\beta}\leq \frac{\kappa^{2(1-2\beta)}}{2\beta-1}<+\infty\,,
\ees
whereas the case $\beta=\nicefrac12$ corresponds to the Gaussian Free Field. For this reason we restrict ourselves to the case $\beta>\nicefrac12$ so that the Mat\'ern field is a random variable taking values in $\LLS$ and hence its law is a probability measure on $\LLS$.

\medskip

Let us take as target noise measure $\frm^{\kappa,\beta}$ the law on $\LLS$ of this random field and consider the associated measure space $(\LLS,\frm^{\kappa,\beta})$. Moreover, in view of~\eqref{eq:ang:power:spectrum} we know that the covariance operator $C^{\kappa,\beta}$ associated to $\frm^{\kappa,\beta}$ is of trace class, the spherical harmonics $(Y_{l,m})_{l\geq0}^{m=-\ell,\dots,+\ell}$ form the corresponding an orthonormal basis of $\rmH=\LLS$ and $C^{\kappa,\beta}Y_{\ell,m}=C_\ell^{\kappa,\beta}Y_{\ell,m}$. Then, the Cameron-Martin space associated to $\frm^{\kappa,\beta}$ corresponds to the space
\bes
\cH_{\frm^{\kappa,\beta}}\coloneqq\biggl\{u=\sum_{\ell,m}a_{\ell,m}Y_{\ell,m}\in\LLS\eqsp\text{ s.t. }\eqsp \sum_{\ell\geq 0}\sum_{m=-\ell}^\ell  a_{\ell,m}^2 (\kappa^2+\ell(\ell+1))^{2\beta}<+\infty\biggr\}\,,
\ees
equipped with the scalar product that for any $u=\sum_{\ell,m}a_{\ell,m}^uY_{\ell,m}$ and $v=\sum_{\ell,m}a_{\ell,m}^vY_{\ell,m}$ in $\cH_{\frm^{\kappa,\beta}}$ reads as
\bes
\langle u,v\rangle_{\cH_{\frm^{\kappa,\beta} }}\coloneqq\sum_{\ell\geq 0}\sum_{m=-\ell}^\ell  a_{\ell,m}^u\,a_{\ell,m}^v (\kappa^2+\ell(\ell+1))^{2\beta},.
\ees

\medskip

Then, the finite-dimensional marginalization up to the frequency level $L>0$  corresponds to the vector of random coefficients $(a_{\ell,m})_{\ell\leq L}^{m=-\ell,\dots,+\ell}$. This means that the forward noising dynamics considered in~\eqref{eq:OU:marginalization:forward} corresponds to the OU dynamics for each coefficient $a_{\ell,m}$ independently of each different channel, \ie, these coefficients evolve independently according to
\bes
\begin{cases}
\De \fwdalm(s)=-\frac12\,\fwdalm(s)+(\kappa^2+\ell(\ell+1))^{-\beta}\,\De B_s\\
\fwdalm(0)=\langle u_{\mathrm{data}},Y_{\ell,m}\rangle_{\LLS}\,,
\end{cases}
\ees
where $u_{\mathrm{data}}\sim\data$ is a sample from our data distribution.

Once the forward dynamics for the $\fwdalm$ has been simulated, we may train a neural network as in \Cref{alg:training} in order to learn an approximate score function $\mathrm{v}^{\theta^\star}$ and then generate the new data samples $ Y^{\mathrm{gen}}_T$ where
\bes
 Y^{\mathrm{gen}}_s\coloneqq \sum_{\ell,m} g_{\ell,m}(s) Y_{\ell,m}\,,
\ees
and where the random coefficients $g_{\ell,m}(t)$ are proxies for $\bwdalm(t)\sim\fwdalm(T-t)$ and are sampled (as in \Cref{alg:sampling}) according to~\eqref{eq:SDE:approx} that for any  $s\in[t_j,t_{j+1}]$ reads as
   \bes
\begin{cases}
    \begin{aligned}\De g_{\ell,m}(s)=-\frac12\,g_{\ell,m}(s)\De s+\sigma_{T-t_j}^{-1}\,(\mathrm{v}^{\theta^\star}_{T-t_j}( Y^{\mathrm{gen}}_{t_j})-e^{\nicefrac{T-t_j}{2}}g_{\ell,m}(t_j))\De s\\
    +(\kappa^2+\ell(\ell+1))^{-\beta}\,\De B_s\end{aligned}\\
    \text{and it starts with }g_{\ell,m}(0)\sim\cN(0,(\kappa^2+\ell(\ell+1))^{-2\beta}) \,.
\end{cases}
\ees
Therefore, once the score is learned, we generate new random fields with law (close to) $\data$ via simulating the above SDE for the $a_{\ell,m}$ coefficients and the SDE is started in the Gaussian centred variables $\cN(0,(\kappa^2+\ell(\ell+1))^{-2\beta})$, that correspond to the coefficients of a Whittle-Mat\'ern field.

\subsubsection{Numerical simulation} 
\begin{figure}[h]
    \centering
    \begin{subfigure}{0.3\textwidth}
        \includegraphics[width=\linewidth]{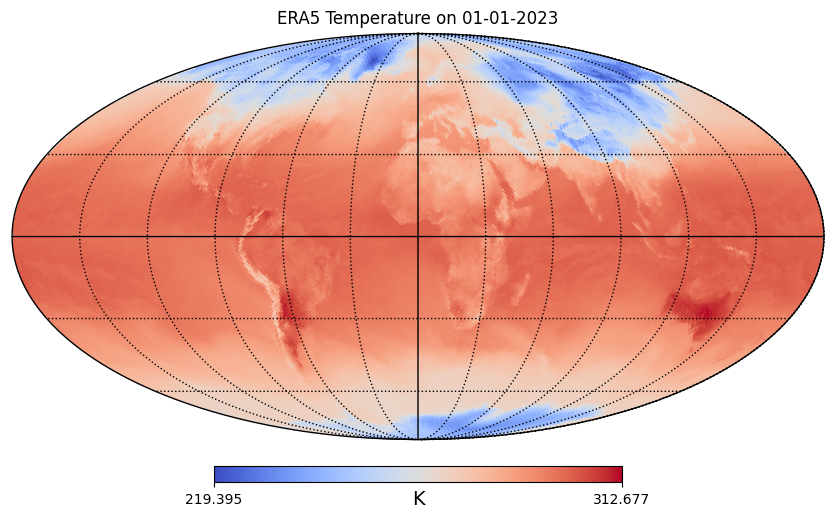}
    \end{subfigure}\hfill
    \begin{subfigure}{0.3\textwidth}
        \includegraphics[width=\linewidth]{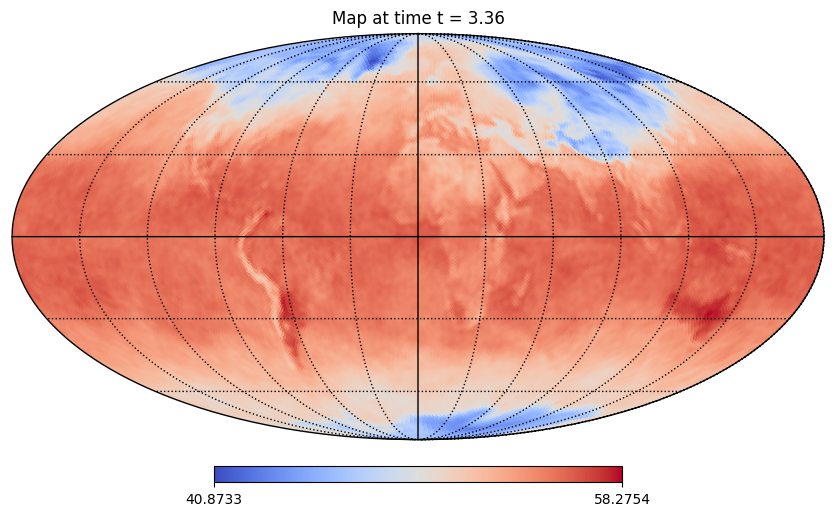}
    \end{subfigure}\hfill
    \begin{subfigure}{0.3\textwidth}
        \includegraphics[width=\linewidth]{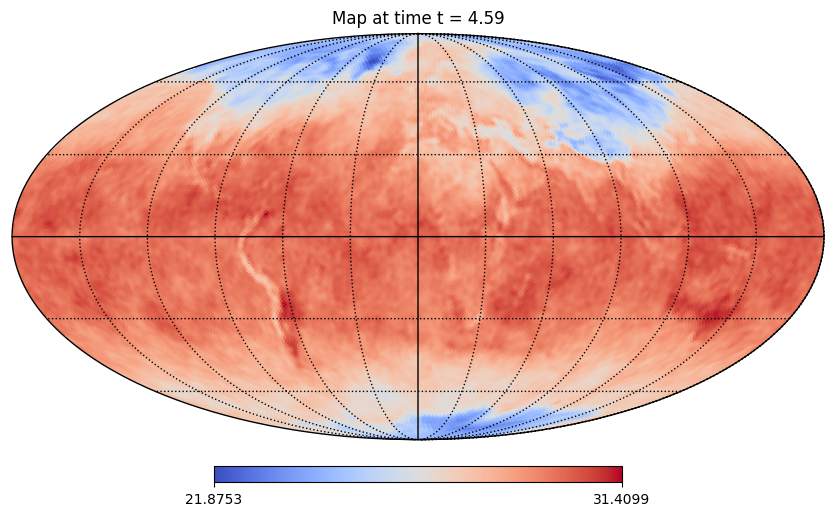}
    \end{subfigure}
    
     \begin{subfigure}{0.3\textwidth}
        \includegraphics[width=\linewidth]{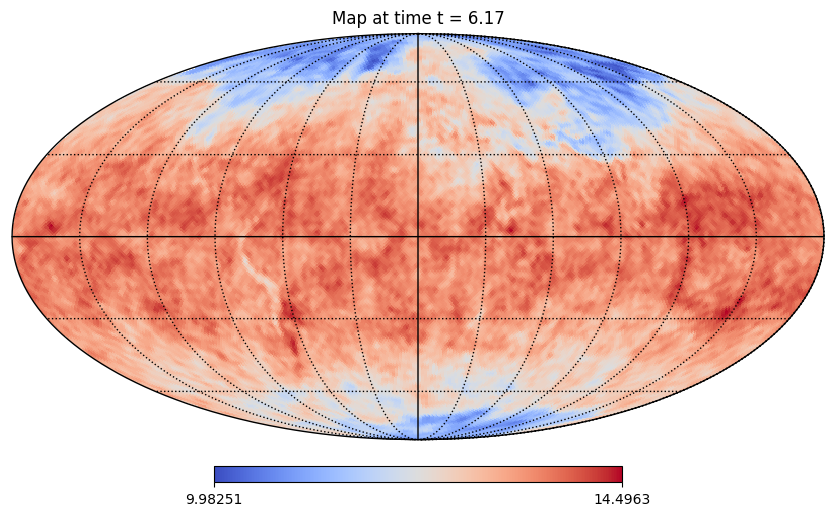}
    \end{subfigure}\hfill
    \begin{subfigure}{0.3\textwidth}
        \includegraphics[width=\linewidth]{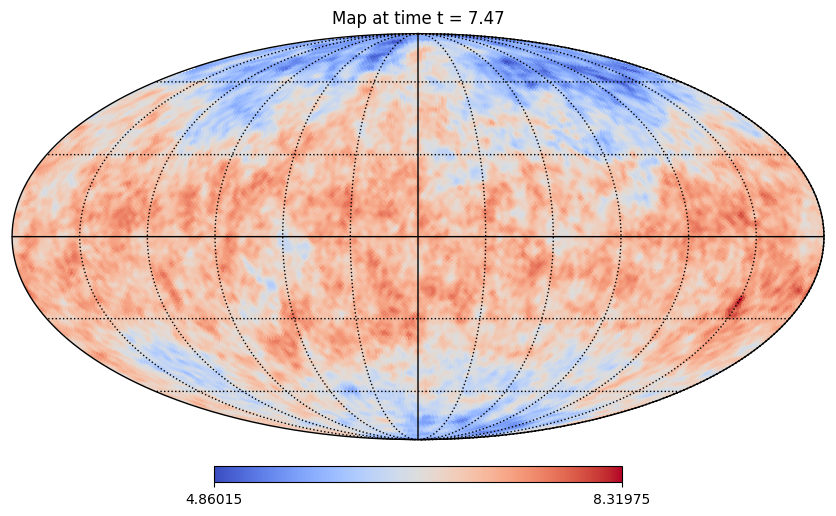}
    \end{subfigure}\hfill
    \begin{subfigure}{0.3\textwidth}
        \includegraphics[width=\linewidth]{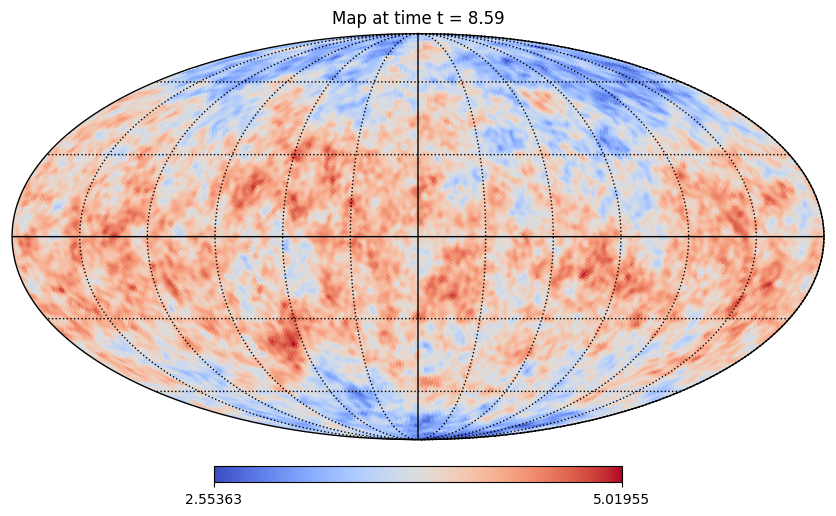}
    \end{subfigure}
    
    \begin{subfigure}{0.3\textwidth}
        \includegraphics[width=\linewidth]{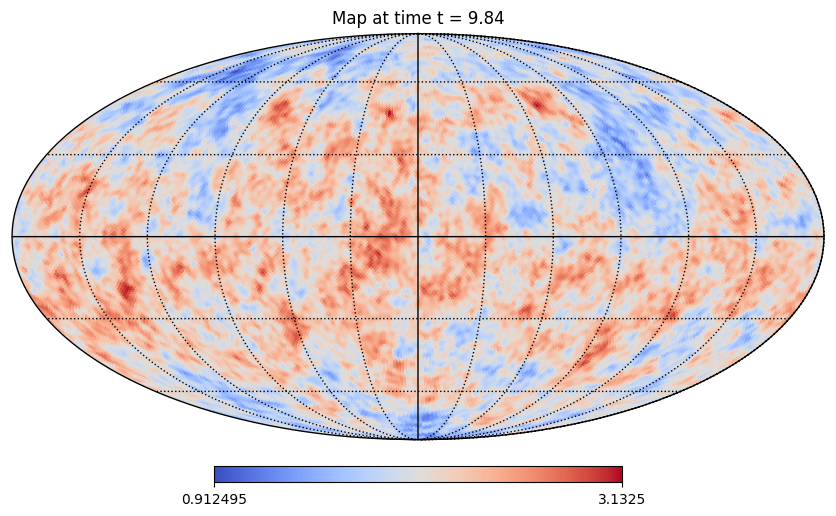}
    \end{subfigure}\hfill
    \begin{subfigure}{0.3\textwidth}
        \includegraphics[width=\linewidth]{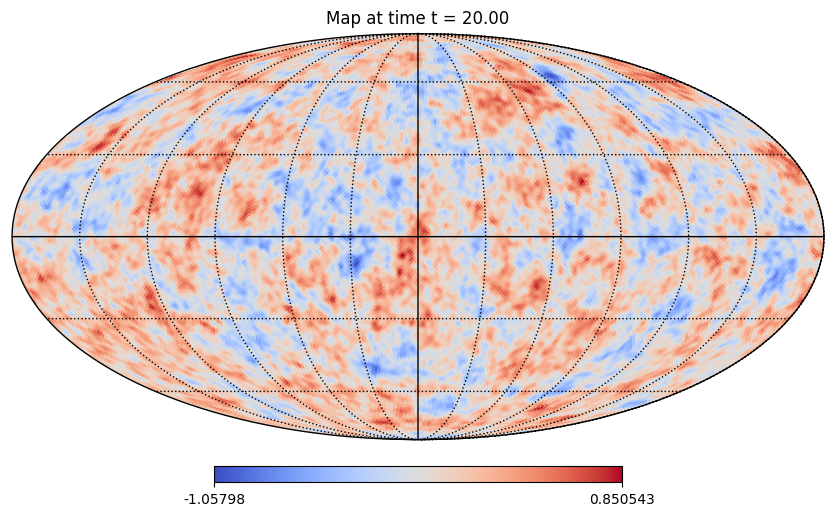}
    \end{subfigure}\hfill
    \begin{subfigure}{0.3\textwidth}
        \includegraphics[width=\linewidth]{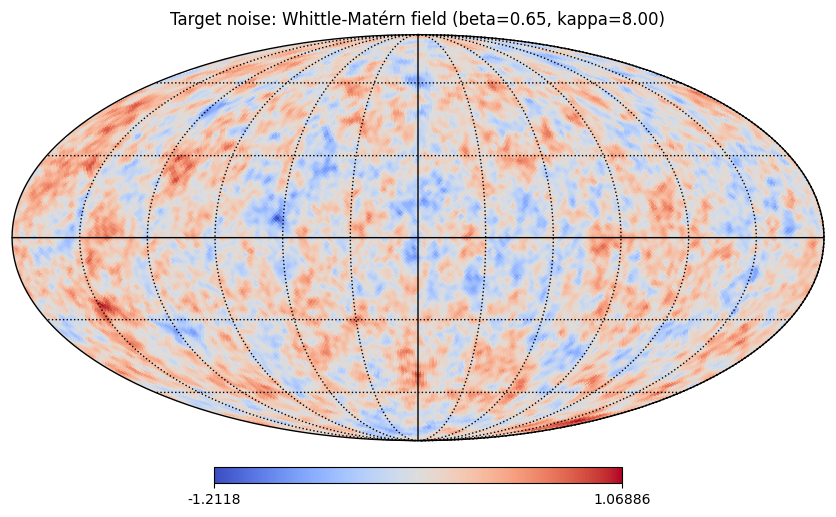}
    \end{subfigure}
    
    \caption{Simulation of the forward noising process at different time-steps.}
    \label{fig:noising}
\end{figure}
In \Cref{fig:noising} we simulate the forward noising dynamics described in our spherical random fields example. In this example, we have considered an ERA5 map from the European Centre for Medium-Range Weather Forecasts \cite{hersbach_era5_2018, copernicus_era5_2023}  plotting the air temperature at 2 meters above the surface of the Earth on the 01/01/2023 (in Kelvin). In order to fit our framework we have firstly turned this ERA5 map into an Healpy map (top left corner in \Cref{fig:noising}), in order to manipulate it with the Python package Healpy useful for spherical random fields simulations, which is based on the HEALPix\footnote{\url{http://healpix.sourceforge.net}} C++ library \cite{Zonca2019, 2005ApJ...622..759G}.  Then, we have computed the (real) harmonic coefficients $(a_{\ell,m})_{\ell,m}$. To be more precise, we have firstly applied the function \emph{healpy.sphtfunc.map2alm}, with $L_{\max}=128$, which returns the complex harmonic coefficients associated to a given Healpy map and then from the latters we have computed the real ones. 

Then, we have applied the noising dynamics  (OU evolution) to the vector of the (real) harmonic coefficients $(a_{\ell,m})_{\ell,m}$ as described in the previous section, with  target noise equal to a Whittle-Mat\'ern field with $\kappa=8$ and $\beta=0.65$ (bottom right corner in \Cref{fig:noising}). For the simulation of the OU process we have explicitly computed the time evolution in $10^3$ time-steps equidistributed in $[0,1]$ and in $10^3$ steps equidistributed in $[1,20]$.

Finally, from the OU evolution of the $(a_{\ell,m})_{\ell,m}$ we have computed the corresponding evolution for the complex harmonic coefficients and built the time-evolution of the Healpy maps using the function \emph{healpy.sphtfunc.alm2map}. The plot has been obtained using \emph{healpy.visufunc.mollview}.

We have stopped our process at $T_{\max}=20$ since at this stage our empirical angular power spectrum $(\hat{C}_\ell)_{\ell\in[0,128]}$ (\ie, the one computed by averaging over $m$ the harmonic coefficients $(a_{\ell,m})_{m=-\ell,\dots,0,\dots,+\ell}$) fits the theoretical angular power spectrum of the  Whittle-Mat\'ern field given at~\eqref{eq:ang:power:spectrum}. In \Cref{fig:empirical:Cl} and \Cref{fig:averaged:Cl} we plot the theoretical angular power spectrum, the empirical one, and an empirical power spectrum where we average over the last $30$ time-steps of the OU evolution (to compensate the error in the low frequencies of $\hat{C}_\ell$, due to the lack of enough harmonic coefficients to average out for low values of $\ell$).  Lastly, in \Cref{fig:delta} we plot the ratio between the empirical (and averaged empirical) $\hat{C}_\ell$ and the theoretical $C_\ell$.

\begin{figure}[h]\label{fig:empirical:Cl}
    \centering
    \begin{subfigure}{0.3\textwidth}
        \includegraphics[width=\linewidth]{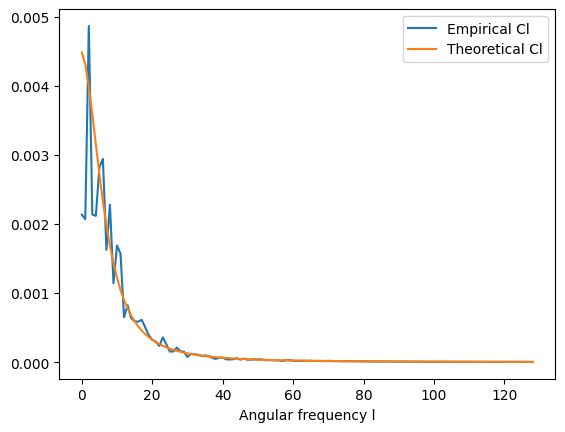}
    \end{subfigure}\hfill
    \begin{subfigure}{0.3\textwidth}
        \includegraphics[width=\linewidth]{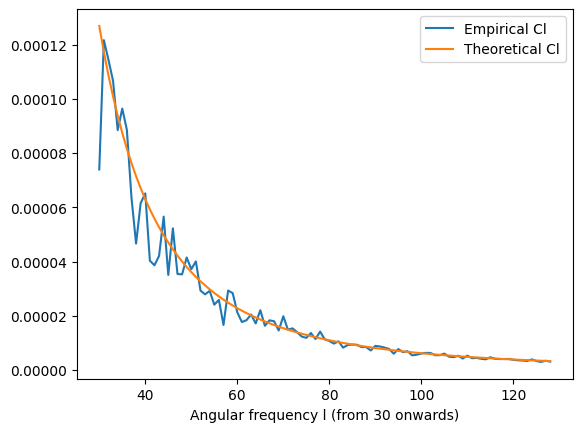}
    \end{subfigure}\hfill
    \begin{subfigure}{0.3\textwidth}
        \includegraphics[width=\linewidth]{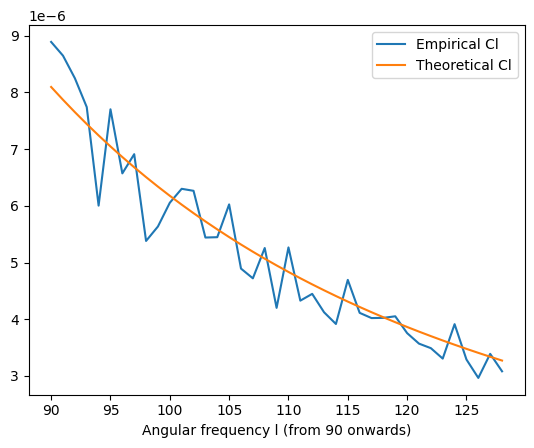}
    \end{subfigure}
    \caption{ Plot of empirical and theoretical power spectra}
\end{figure}

\begin{figure}[h]\label{fig:averaged:Cl}
    \centering
    \begin{subfigure}{0.3\textwidth}
        \includegraphics[width=\linewidth]{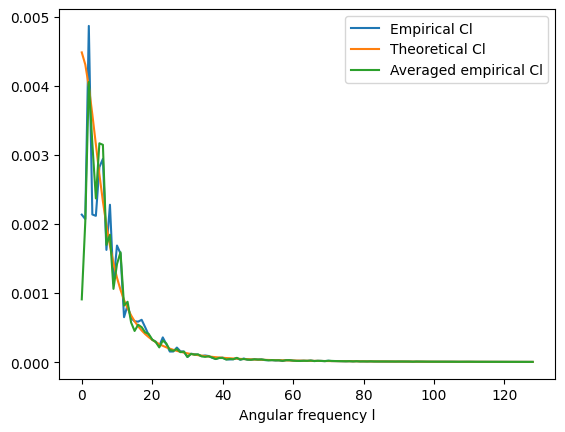}
    \end{subfigure}\qquad  \qquad
    \begin{subfigure}{0.3\textwidth}
        \includegraphics[width=\linewidth]{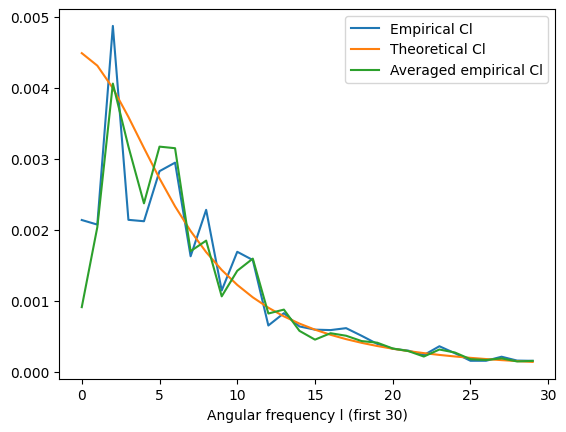}
    \end{subfigure}
    \caption{Plot of averaged (over last 30 steps), empirical and theoretical power spectra}
\end{figure}

\begin{figure}[h]\label{fig:delta}
    \centering
    \begin{subfigure}{0.3\textwidth}
        \includegraphics[width=\linewidth]{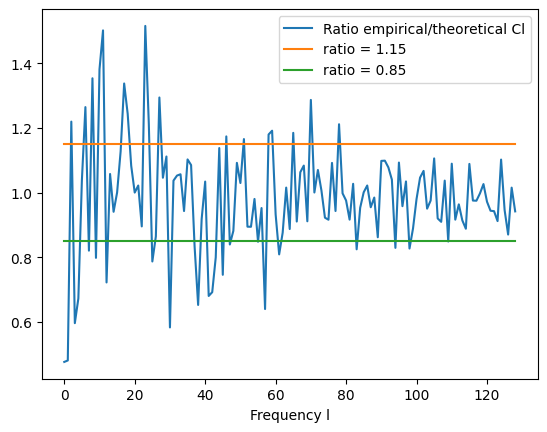}
    \end{subfigure}\qquad  \qquad
    \begin{subfigure}{0.3\textwidth}
        \includegraphics[width=\linewidth]{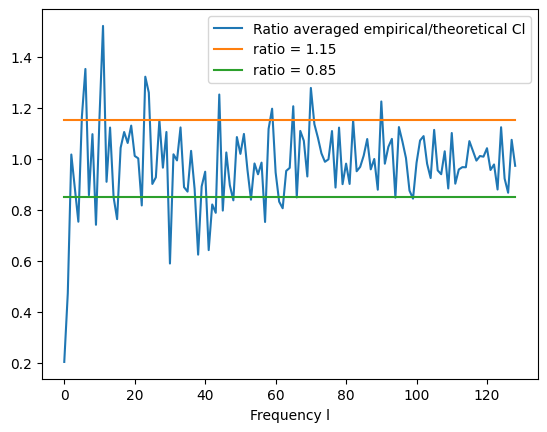}
    \end{subfigure}
    \caption{Ratio between empirical $\hat{C}_\ell$ and theoretical $C_\ell$ (on the left) and ratio between averaged empirical $\hat{C}_\ell$ and theoretical $C_\ell$ (on the right)}
\end{figure}

\medskip

Given the above preprocessing of the ERA5 maps into Healpy maps and real vector of harmonic coefficients, the training algorithm described in \Cref{alg:training} can be obtained by adapting any standard architecture already employed in the literature for vector-valued diffusion models, taking into account the OU dynamics we have choosen. Once this is done, \Cref{alg:sampling} is straightforward. Implementing the generative model was out of the scope of the current paper, which, on the other hand, investigates the theoretical point of view of infinite-dimensional-generative models and analyses theoretical convergence bounds (cf.~\Cref{sec:convergence_bounds}). 

The purpose of the current section is merely visualizing the noising forward dynamics of the harmonic coefficients in terms of spherical fields, in order to help the target readership (unfamiliar with generative models) in understanding what happens along the noising algorithm. 
In view of that, let us conclude by mentioning that the noising forward procedure (which takes places along \Cref{alg:training}) should be thought as following \Cref{fig:noising} from the top-left corner towards the bottom-right corner, whereas the denoising procedure (\ie, the sampling algorithm \Cref{alg:sampling}) can be seen as following \Cref{fig:noising} from the bottom-right corner  towards the top-left corner.


\section{Convergence results for finite-dimensional approximations}\label{sec:convergence_bounds}
In this section we provide a $\scrH$-convergence bound for the generative model, when considering a fixed finite-dimensional marginalization of our random field. The following is a direct adaptation of \cite{Conforti2025KLConvergence} to our setting. For this reason, we will omit details and just sketch how their result adapts/applies to ours. Our goal is bounding the $\scrH$-divergence between (finite-dimensional marginalizations of) $\data$ and the distribution generated via~\eqref{eq:SDE:approx}, \ie, the law of $Y_T$. Below we won't take into account the error due to the application of Euler-Maruyama scheme in \Cref{alg:sampling}, which can be taken into account via standard results, and we will focus on the propagation of two errors involving the process $Y_\cdot$ (which should mimic the backward process $\bwdX_\cdot$):
\begin{itemize}
    \item a first error is due to the fact that $Y_0\sim\frm$ is sampled from the Gaussian measure, however $\bwdX_0\sim\bwdP_0=\fwdP_T$ which is not the Gaussian,
    \item a second one due to the fact that we have access to an approximated version of the score $\mathrm{v}^{\theta^\star}$ and moreover it is defined on the time grid $(T-t_k)_{k=0,\dots,M}$ and not on the whole time interval $[0,T]$.
\end{itemize}

In order to do that, we will rely on Girsanov's theory \cite{Leonard2012Girsanov} for finite-dimensional diffusions and on the projections considered in the proof of \Cref{cor:back:SDE}. Hence, fix $N\in\N$, a finite collection $J=\{\ell_1,\dots,\ell_N\}$ and consider the finite-dimensional backward process $\bwdX^J_\cdot$ weakly solving (cf.~proof of \Cref{cor:back:SDE})
\bes
\begin{cases}
    \De \bwdX_s^J=-\frac{1}{2}\,\bwdX^J_s\De s+\sigma_{T-s}^{-1}\,\biggl(\bbE_{\fwdP^J}[\fwdX^J_0|\fwdX^J_{T-s}=\bwdX^J_s]-e^{\nicefrac{T-s}{2}}\bwdX_s^J\biggr)\De s+\Sigma^J\,\De B_s\\
    \bwdX_0^J\sim\fwdP^J_T\,.
\end{cases}
\ees
By mimicking the scheme introduced in the previous section we can then take $\mathrm{v}^{\theta^\star,J}_s$ as the minimizer in \eqref{eq:likelihood:scores} (this time evaluated for the finite-dimensional forward process $\fwdX_\cdot^J$) and consider the corresponding approximate backward process
\bes
\begin{cases}
    \De Y_s^J=-\frac{1}{2}\,Y_s^J\De s+\sigma_{T-t_k}^{-1}\,(\mathrm{v}^{\theta^\star,J}_{T-t_k}(Y_{t_k}^J)-e^{\nicefrac{T-t_k}{2}}Y^J_{t_k})\De s+\Sigma^J\,\De B_s \quad \forall\,s\in[t_k,t_{k+1}]\\
    Y_0^J\sim\frm^J=\mathcal{N}(0,(\Sigma^J)^2)\,,
\end{cases}
\ees
with $\Sigma^J\coloneqq\mathrm{diag}(C_{\ell_j}^{\nicefrac12})_{j=1,\dots,N}$. Finally we are going to assume that our approximate score $\mathrm{v}^{\theta^\star,J}$ is good enough, that is
\begin{assumption}\label{ass:score:goodness}
    there exists $\varepsilon>0$ and $\theta^\star\in\Theta$ such that
    \bes
\frac1M\,\sum_{j=1}^M\,\bbE\norm{\bbE[\fwdX^J_0|\fwdX^J_{T-t_j}=\bwdX^J_{t_j}]-\mathrm{v}_{T-t_j}^{\theta^\star,J}(\bwdX^{J}_{t_j})}_{(\Sigma^J)^{-1}}^2 \leq \varepsilon^2\,,
    \ees
    where the norm induced by $\Sigma^J$ is defined for any $x\in\rset^N$ as $\norm{x}_{(\Sigma^J)^{-1}}\coloneqq \norm{(\Sigma^J)^{-1}\,x}$.
\end{assumption}

Assumptions like this one (with the usual Euclidean norm) are commonly considered in the literature of score-based generative models \cite{lee2022convergence,pmlr-v201-lee23a,Conforti2025KLConvergence,strasman2025an}. The presence of the weighted norm $\norm{\cdot}_{(\Sigma^J)^{-1}}$ is due to the fact that in the limit the actual norm that plays a role in the convergence bound is the norm induced by the Cameron-Martin space. This means that any attempt of extending the result presented here to the law of the whole random fields should replace \Cref{ass:score:goodness} with
 \bes
\frac1M\,\sum_{j=1}^M\,\bbE\norm{\bbE[\fwdX_0|\fwdX_{T-t_j}=\bwdX_{t_j}]-\mathrm{v}_{T-t_j}^{\theta^\star}(\bwdX_{t_j})}_{\cH_\frm}^2 \leq \varepsilon^2\,.
    \ees

\medskip

We are now ready to prove our $\scrH$-convergence bound.

\begin{theorem}\label{thm:KL:convergence:bound}
     Assume $\scrH(\data|\frm)$ and $I(\data|\frm)$ to be finite, fix a finite-dimensional marginalization $J=\{\ell_1,\dots,\ell_N\}$, a time-discretization step $h>0$ and assume \Cref{ass:score:goodness}. Then we have 
     \bes
\scrH(\data^J|\cL(Y_T^J))
    \leq e^{-\nicefrac{T}{2}}\,\scrH(\data|\frm)+T\,\varepsilon^2+2\,h\,\max\{4,h\} \,\bbE_{\frm}\norm{\bar\nabla\sqrt{\rho_0}}^2_{\cH_\frm}\,.
\ees
\end{theorem}
\begin{proof}
 Since $\data^J=\cL(\bwdX_T^J)$, from \cite[Theorem 2.3]{Leonard2012Girsanov} and from the data-processing inequality \cite[Lemma 1.6]{Marcel:notes}, we immediately deduce that
 
 \be\label{eq:bound:3:errori}\begin{aligned}
   \,& \scrH(\data^J|\cL(Y_T^J))
    \leq\,\scrH(\cL(\bwdX^J_\cdot)|\cL(Y^J_\cdot))\\
    =&\,\scrH(\cL(\bwdP_0^J)|\cL(Y^J_0))\\
 &\quad+\frac12\sum_{k=0}^{M-1}\int_{t_k}^{t_{k+1}} \bbE_{\bwdP^J}\norm{V^J_s(\bwdX^J_s)-\sigma_{T-t_k}^{-1}(\mathrm{v}^{\theta^\star,J}_{T-t_k}(\bwdX^J_{t_k})-e^{\nicefrac{T-t_k}{2}}\bwdX_{t_k}^J)}_{(\Sigma^J)^{-1}}^2\De s\\
    \leq&\,\scrH(\fwdP_T^J|\frm^J)\\
    &\quad+\sum_{k=0}^{M-1}\int_{t_k}^{t_{k+1}} \bbE_{\bwdP^J}\norm{V^J_{t_k}(\bwdX^J_{t_k})-\sigma_{T-t_k}^{-1}(\mathrm{v}^{\theta^\star,J}_{T-t_k}(\bwdX^J_{t_k})-e^{\nicefrac{T-t_k}{2}}\bwdX_{t_k}^J)}_{(\Sigma^J)^{-1}}^2\De s\\
    &\quad +\sum_{k=0}^{M-1}\int_{t_k}^{t_{k+1}} \bbE_{\bwdP^J}\norm{V^J_s(\bwdX^J_s)-V^J_{t_k}(\bwdX^J_{t_k})}_{(\Sigma^J)^{-1}}^2\De s\,,
    \end{aligned}\ee
    where for notations' sake we have introduced the (finite-dimensional relative) score function (see for instance \cite[Lemma 1]{Pidstrigach2024Infinite}) 
    \bes
    V_s^J(x)\coloneqq \sigma_{T-s}^{-1}\biggl(\bbE_{\fwdP^J}[\fwdX^J_0|\fwdX^J_{T-s}=x]-e^{\nicefrac{T-s}{2}}x\biggr)=C^J\nabla\log\rho^J_{T-s}(x)\,,
    \ees
    with $C^J=(\Sigma^J)^2=(C_{\ell_j})_{j=1,\dots,N}$ and $\rho^J_s\coloneqq \frac{\De \fwdP^J_s}{\De \frm^J}$ where $\frm^J$ is the $N$-dimensional centered Gaussian measure with covariance $(\Sigma^J)^2$ (\ie, $\rho^J$ is the density w.r.t. the equilibrium of the law of $\fwdX^J$).

    The first term in the last line of~\eqref{eq:bound:3:errori} can be bounded via Log-Sobolev inequality \cite[Theorem 5.5.1]{Bogachev1998book} (see also \cite[Theorem 5.2.1 and Proposition 5.2.2]{bakry2013analysis}). Indeed, the former implies for any $t\in[0,T]$
    \bes
\scrH(\fwdP_t|\frm)\leq 2\int\norm{\bar{\nabla}\sqrt{\rho_t}}_{\cH_\frm}^2\De \frm=4\,\cE(\sqrt{\rho_t})\,,
\ees
which combined with the latter gives
\bes
\begin{aligned}
\frac{\De}{\De t}\scrH(\fwdP_t|\frm)=\frac{\De}{\De t}\int(P_t\rho_0)\log(P_t\rho_0)\De\frm=\int(1+\log(P_t\rho_0))L(P_t\rho_0)\De \frm\\     
=-4\int\frac{\Gamma(\rho_t)}{\rho_t}\De\frm
=-4\int\norm{\bar{\nabla}\sqrt{\rho_t}}_{\cH_\frm}^2\De \frm\leq-\frac12\scrH(\fwdP_t|\frm)\,.
\end{aligned}
\ees
Therefore the first term can be bounded via the data processing inequality for relative entropy and Gronwall's lemma as 
\be\label{eq:LSI:conv:KL}
\scrH(\fwdP_T^J|\frm^J)\leq \scrH(\fwdP_T|\frm)\leq e^{-\nicefrac{T}{2}}\,\scrH(\fwdP_0|\frm)=e^{-\nicefrac{T}{2}}\,\scrH(\data|\frm)\,.
\ee

The second term from~\eqref{eq:bound:3:errori} is bounded by $T\varepsilon^2$, thanks to our assumption \Cref{ass:score:goodness}. 

Lastly, let us focus on the last term in~\eqref{eq:bound:3:errori}. 
 In order to do that, hereafter let us assume $\data$ to have finite eight moments, this extra assumption can be lifted via  a standard approximation procedure as in \cite[Lemma 3]{Conforti2025KLConvergence}.  Next, notice that 
by reasoning as in \cite[Proposition 2 and Equation (27)]{Conforti2025KLConvergence} we can deduce (see \Cref{lemma:gio}) that 
\be\label{SDE:score:process}
\De V^J_t(\bwdX_t^J)=\frac12 V^J_t(\bwdX_t^J)\De t+\Sigma^J\nabla V^J_t(\bwdX_t^J)\,\De B_t\,,
\ee
and 
\be\label{SDE:square:score}
    \begin{aligned}
\De \norm{V^J_t(\bwdX^J_t)}_{(\Sigma^J)^{-1}}^2=\biggl(\norm{V^J_t(\bwdX^J_t)}_{(\Sigma^J)^{-1}}^2+\norm{\nabla V^J_t(\bwdX^J_t)}^2_{\mathrm{Fr}}\biggr)\De t\\
+2\, V^J_t(\bwdX^J_t)\,\cdot(\Sigma^J)^{-1}\nabla V^J_t(\bwdX_t^J)\De B_t\,,
 \end{aligned}
\ee
with $\norm{\cdot}_{\mathrm{Fr}}$ denoting the Frobenious norm.
     By integrating~\eqref{SDE:score:process}, from Jensen's inequality  we may deduce that for any $s\in[t_k,t_{k+1}]$ it holds
\bes
\begin{aligned}
    \bbE_{\bwdP^J}&\norm{V^J_s(\bwdX^J_s)-V^J_{t_k}(\bwdX^J_{t_k})}_{(\Sigma^J)^{-1}}^2\\
    =&\,\bbE_{\bwdP^J}\norm{\frac12 \int_{t_k}^sV^J_u(\bwdX_u^J)\De u+\int_{t_k}^s\Sigma^J\nabla V^J_u(\bwdX_u^J)\,\De B_u}_{(\Sigma^J)^{-1}}^2\\
    \leq&\, \frac{s-t_k}2\int_{t_k}^s \bbE_{\bwdP^J}\norm{V^J_u(\bwdX^J_u)}_{(\Sigma^J)^{-1}}^2\De u+2\int_{t_k}^s \bbE_{\bwdP^J}\norm{\nabla V^J_u(\bwdX^J_u)}^2_{\mathrm{Fr}}\De u\\
    \leq &\,\max\{2,\nicefrac{h}{2}\}\,\int_{t_k}^{t_{k+1}}\bbE_{\bwdP^J}\norm{V^J_u(\bwdX^J_u)}_{(\Sigma^J)^{-1}}^2+\bbE_{\bwdP^J}\norm{\nabla V^J_u(\bwdX^J_u)}^2_{\mathrm{Fr}}\,\De u\,,
    \end{aligned}
    \ees
    where in the last step we have notice that $s-t_{k}\leq t_{k+1}-t_k=\nicefrac{T}{N}=h$.
    By integrating~\eqref{SDE:square:score}, since under the finite eight moment assumption  $ \int_{t_k}^\cdot V^J_t(\bwdX^J_t)\,\cdot(\Sigma^J)^{-1}\nabla V^J_t(\bwdX_t^J)\De B_t$ is a martingale, for any $s\in[t_k,t_{k+1}]$ we finally deduce that
    \bes
    \begin{aligned}
    \bbE_{\bwdP^J}&\norm{V^J_s(\bwdX^J_s)-V^J_{t_k}(\bwdX^J_{t_k})}_{(\Sigma^J)^{-1}}^2\\
    &\qquad\qquad\leq \max\{2,\nicefrac{h}{2}\}\,\biggl(\bbE_{\bwdP^J}\norm{V^J_{t_{k+1}}(\bwdX^J_{t_{k+1}})}_{(\Sigma^J)^{-1}}^2-\bbE_{\bwdP^J}\norm{V^J_{t_{k}}(\bwdX^J_{t_{k}})}_{(\Sigma^J)^{-1}}^2\biggr)\,,
    \end{aligned}
    \ees
    and hence that the third term in~\eqref{eq:bound:3:errori} can be bounded as follows
    \be\label{appo:bound:fisher:finita}
\begin{aligned}
    \sum_{k=0}^{M-1}\int_{t_k}^{t_{k+1}} &\bbE_{\bwdP^J}\norm{V^J_s(\bwdX^J_s)-V^J_{t_k}(\bwdX^J_{t_k})}_{(\Sigma^J)^{-1}}^2\De s\\
    \leq &\,h\,\max\{2,\nicefrac{h}{2}\} \sum_{k=0}^{M-1} \bbE_{\bwdP^J}\norm{V^J_{t_{k+1}}(\bwdX^J_{t_{k+1}})}_{(\Sigma^J)^{-1}}^2-\bbE_{\bwdP^J}\norm{V^J_{t_{k}}(\bwdX^J_{t_{k}})}_{(\Sigma^J)^{-1}}^2\\
    =&\,h\,\max\{2,\nicefrac{h}{2}\} \biggl(\bbE_{\bwdP^J}\norm{V^J_T(\bwdX^J_{T})}_{(\Sigma^J)^{-1}}^2-\bbE_{\bwdP^J}\norm{V^J_0(\bwdX^J_0)}_{(\Sigma^J)^{-1}}^2\biggr)\\
    \leq &\,h\,\max\{2,\nicefrac{h}{2}\} \,\bbE_{\bwdP^J}\norm{V^J_T(\bwdX^J_{T})}_{(\Sigma^J)^{-1}}^2\\
    =&\,h\,\max\{2,\nicefrac{h}{2}\} \,\bbE_{\fwdP^J}\norm{\Sigma^J\nabla\log\rho^J_0(\fwdX_0^J)}^2_2,
\end{aligned}
\ee
where in the last step we have used the definition of $V^J_s$ and the time-reversal identity. Next, since $\rmP^J_0=\data^J=(\mathrm{proj}_J)_{\#}\data$ and $\mathrm{proj}_J=I((C_{\ell_j}Y_{\ell_j})_{j=1,\dots,N})$  for any $\Phi_J\in\rset^N$ it holds
\bes
\rho^J_0(\Phi_J)=\data(\mathrm{proj}_J^{-1}(\Phi_J))=\int \rho_0([\Phi_J,\Phi_{J^\perp}]_\rmH)\De \frm^{J^\perp}( \Phi_{J^\perp})\,,
\ees
where $[\Phi_J,\Phi_{J^\perp}]_{\rmH}\in\rmH$ denotes the element whose coordinates are $(\Phi_J,\Phi_{J^\perp})$ and
where we have defined $\Phi^{J^\perp}$ as the projection of $\Phi\in\rmH$ onto the orthogonal complement to $\mathrm{Span}\{Y_{\ell_j}\colon j=1,\dots,N\}$, that is the subspace generated by the element of the orthonormal basis not in $J$, whereas $\frm^{J^\perp}$ is simply the pushforward under this orthogonal projection of $\frm$. 
If $e_j$ denotes the $j^{th}$ element of the standard basis for $\rset^N$, for any $j=1,\dots,N$ it holds 
\bes
\begin{aligned}
\partial_j\rho_0^J(\Phi_J)=\De_{t_{|t=0}}\rho_0^J(\Phi_J+t \,e_j)=\De_{t_{|t=0}}\int \rho_0([\Phi_J+t\,e_j,\Phi_{J^\perp}]_\rmH)\De \frm^{J^\perp}( \Phi_{J^\perp})\\
=\int \De_{t_{|t=0}}\rho_0([\Phi_J,\Phi_{J^\perp}]_\rmH+t\,Y_{\ell_j})\De \frm^{J^\perp}( \Phi_{J^\perp})\,,
\end{aligned}\ees
where the commutation between differentiation and integration will follow from the bound we provide below, which can be generalised to show that $\De_t\rho_0([\Phi_J,\cdot]_\rmH+t\,Y_{\ell_j})\in\rmL^1(\frm^{J^\perp})$ holds $\frm^J$-a.e.  (since $t\,Y_{\ell_j}\in\cH_\frm$).

 Therefore, from Jensen's inequality (applied to the jointly convex function $(a,b)\mapsto \nicefrac{a^2}{b}$) for any $j=1,\dots,\,N$ we have
\bes\begin{aligned}
(\partial_j\sqrt{\rho_0^J})^2(\Phi_J)=&\,\frac{(\partial_j\rho_0^J)^2}{4\,\rho^J_0}(\Phi_J)=\frac{(\De_{t_{|t=0}}\rho_0^J(\Phi_J+t\,e_j))^2}{4\,\rho^J_0(\Phi_J)}\\
=&\,\frac{\biggl(\int \De_{t_{|t=0}}\rho_0([\Phi_J,\Phi_{J^\perp}]_\rmH+t\,Y_{\ell_j})\De \frm^{J^\perp}( \Phi_{J^\perp})\biggr)^2}{\int 4\,\rho_0([\Phi_J,\Phi_{J^\perp}]_\rmH)\De \frm^{J^\perp}( \Phi_{J^\perp})}\\
\leq&\, \int \frac{(\De_{t_{|t=0}}\rho_0([\Phi_J,\Phi_{J^\perp}]_\rmH+t\,Y_{\ell_j})\,)^2}{4\,\rho_0([\Phi_J,\Phi_{J^\perp}]_\rmH)}\De \frm^{J^\perp}( \Phi_{J^\perp})\\
=&\,\int \biggl(\De_{t_{|t=0}}\sqrt{\rho_0}([\Phi_J,\Phi_{J^\perp}]_\rmH+t\,Y_{\ell_j})\biggr)^2\De \frm^{J^\perp}( \Phi_{J^\perp})\\
=&\,\int \langle\bar\nabla\sqrt{\rho_0}([\Phi_J,\Phi_{J^\perp}]_\rmH),\,Y_{\ell_j}\rangle_{\cH_\frm}^2\,\De \frm^{J^\perp}( \Phi_{J^\perp})\,.
\end{aligned}\ees
From this we may finally deduce that
 \be\label{appo:bound:Fisher:fin2infinity}
\begin{aligned}
    \bbE_{\fwdP^J}\norm{\Sigma^J\nabla\log\rho^J_0(\fwdX_0^J)}&^2_2
    =\,\bbE_{\data^J}\norm{\frac{\Sigma^J\nabla\rho^J_0}{\rho^J_0}}^2_2=4\,\bbE_{\frm^J}\norm{\Sigma^J\nabla\sqrt{\rho^J_0}}^2_2\\
    =&\,4\,\bbE_{\frm^J}\biggl[\sum_{j=1}^N C_{\ell_j}\,(\partial_j\sqrt{\rho^J_0})^2\biggr]\\
    \leq&\, 4\,\bbE_{\frm^J}\biggl[\sum_{j=1}^N C_{\ell_j}\,\int \langle\bar\nabla\sqrt{\rho_0}([\Phi_J,\Phi_{J^\perp}]_\rmH),\,  Y_{\ell_j}\rangle_{\cH_\frm}^2\,\De \frm^{J^\perp}( \Phi_{J^\perp})\biggr]\\
    =&\,4\,\bbE_{\Phi\sim\frm}\biggl[ \sum_{j=1}^N \langle\bar\nabla\sqrt{\rho_0}(\Phi),\,C_{\ell_j}^{\nicefrac12}\,Y_{\ell_j}\rangle_{\cH_\frm}^2\biggr]\leq 4\,\bbE_{\frm}\norm{\bar\nabla\sqrt{\rho_0}}^2_{\cH_\frm}\,,
\end{aligned}
\ee
where the last step follows from Bessel's inequality since the collection $\{C_\ell^{\nicefrac12}Y_\ell\}_{\ell\geq 0}$ is orthonormal in the Cameron-Martin space $\cH_\frm$.

By combining \eqref{appo:bound:fisher:finita} and \eqref{appo:bound:Fisher:fin2infinity} we deduce that
\bes
 \sum_{k=0}^{M-1}\int_{t_k}^{t_{k+1}} \bbE_{\bwdP^J}\norm{V^J_s(\bwdX^J_s)-V^J_{t_k}(\bwdX^J_{t_k})}_{(\Sigma^J)^{-1}}^2\De s\\
    \leq 2\,h\,\max\{4,h\} \,\bbE_{\frm}\norm{\bar\nabla\sqrt{\rho_0}}^2_{\cH_\frm}\,.
\ees

In conclusion, from~\eqref{eq:bound:3:errori}, \eqref{eq:LSI:conv:KL}, \Cref{ass:score:goodness} and from this last bound we finally get
\bes
\scrH(\data^J|\cL(Y_T^J))
    \leq e^{-\nicefrac{T}{2}}\,\scrH(\data|\frm)+T\,\varepsilon^2+2\,h\,\max\{4,h\} \,\bbE_{\frm}\norm{\bar\nabla\sqrt{\rho_0}}^2_{\cH_\frm}\,,
\ees
which corresponds to our thesis.
\end{proof}
For completeness, below we prove \eqref{SDE:score:process} and  \eqref{SDE:square:score}.

\begin{lemma}\label{lemma:gio}
For any fixed $J=(\ell_1,\dots,\ell_N)$ the score process $V^J_t(\bwdX^J_t)$ solves
 \bes
\De V^J_t(\bwdX_t^J)=\frac12 V^J_t(\bwdX_t^J)\De t+\Sigma^J\nabla V^J_t(\bwdX_t^J)\,\De B_t\,,
\ees
with $B_\cdot$ being an $N$-dimensional Brownian motion. In particular its square satisfies
\bes\begin{aligned}
\De \norm{V^J_t(\bwdX^J_t)}_{(\Sigma^J)^{-1}}^2=\biggl(\norm{V^J_t(\bwdX^J_t)}_{(\Sigma^J)^{-1}}^2+\norm{\nabla V^J_t(\bwdX^J_t)}^2_{\mathrm{Fr}}\biggr)\De t\\
+2\, V^J_t(\bwdX^J_t)\,\cdot(\Sigma^J)^{-1}\nabla V^J_t(\bwdX_t^J)\De B_t\,,
\end{aligned}
\ees
with $\norm{\cdot}_{\mathrm{Fr}}$ denoting the Frobenious norm.
\end{lemma}
\begin{proof} 
The following proof is an adaptation of~\cite[Proposition 2 and Equation (27)]{Conforti2025KLConvergence} with diffusivity matrix $\Sigma^J$. Since $\rho^J_t$ satisfies
\bes
\partial_t\rho^J=-\frac12\langle x,\,\nabla\rho^J_t\rangle+\frac12\,\Delta_{C^J}\rho^J_t\,,
\ees
with $\Delta_{C^J}$ being the Laplacian operator induced by $C^J$, \ie, $\Delta_{C^J}f\coloneqq \nabla\cdot(C^J\nabla f)$, from Ito's formula we see that
     \bes
\De V^J_t(\bwdX_t^J)=\frac12 V^J_t(\bwdX_t^J)\De t+\Sigma^J\nabla V^J_t(\bwdX_t^J)\,\De B_t\,,
\ees
which implies 
\bes
\begin{aligned}
\De &\norm{V^J_t(\bwdX^J_t)}_{(\Sigma^J)^{-1}}^2=2\langle V^J_t(\bwdX^J_t),\,(C^J)^{-1}\De V^J_t(\bwdX_t^J)\rangle+\Tr[(\nabla V^J_t(\bwdX^J_t))^2] \De t\\
&= \biggl(\norm{V^J_t(\bwdX^J_t)}_{(\Sigma^J)^{-1}}^2+\norm{\nabla V^J_t(\bwdX^J_t)}^2_{\mathrm{Fr}}\biggr)\De t+2\, V^J_t(\bwdX^J_t)\,\cdot(\Sigma^J)^{-1}\nabla V^J_t(\bwdX_t^J)\De B_t\,.
\end{aligned}
\ees
\end{proof}


%
%
\bibliographystyle{alpha}
\bibliography{X2main}
%

\end{document}